\documentclass{article}
\usepackage{a4wide}

\usepackage{amsmath,amssymb,trimclip,adjustbox}%%
\usepackage{graphicx}
\usepackage{authblk}
\usepackage{latexsym}
\usepackage{charter}
\usepackage{tikz}
\usepackage{lineno}
\usepackage{graphicx}

\usepackage{xspace}

\usepackage{latexsym}
\usepackage{charter}
\usepackage{tikz}
\usepackage{subfigure}
\usepackage{float}
\usepackage{xspace}
\usepackage{stmaryrd}

\usepackage{proof}
\usepackage{url}
\usepackage[all]{xy}

\allowdisplaybreaks[4]
\newtheorem{theorem}{Theorem}
\newtheorem{corollary}[theorem]{Corollary}
\newtheorem{lemma}[theorem]{Lemma}
\newtheorem{proposition}[theorem]{Proposition}
\newtheorem{definition}[theorem]{Definition}
\newtheorem{example}{Example}

\newtheorem{remark}{Remark}

\newenvironment{proof} {\textsc{Proof}\quad} {\hfill $\blacksquare$\\}

\usepackage{proof}
\usepackage{url}

\usepackage[all]{xy}

\newcommand{\VDash}{%
  \mathrel{
    \text{\clipbox{0pt 0pt {.8\width} 0pt}{$\Vdash$}}
    \mkern.9mu
    \text{\adjustbox{width=.87\width,height=\height}{$\vDash$}}
  }
}

\newcommand{\CPL}{\mathbf{CPL}}

\newcommand{\M}{\mathcal{M}}

\newcommand{\N}{\mathcal{N}}

\newcommand{\la}{\langle}
\newcommand{\ra}{\rangle}
\newcommand{\lr}[1]{\la #1 \ra}
\newcommand{\lra}{\leftrightarrow}

\renewcommand{\phi}{\varphi}

\newcommand{\SK}{\mathbf{K}}

\newcommand{\Kh}{\mathsf{K\!h }}
\newcommand{\K}{\mathsf{K}}
\newcommand{\hK}{\widehat{\K}}
\newcommand{\hKh}{\widehat{\Kh}}
\newcommand{\DELKh}{\mathbf{DELKh}}

\newcommand{\Prop}{\mathbf{P}}
\newcommand{\PL}{\mathbf{PL}}
\newcommand{\InqL}{\mathbf{InqB}}
\newcommand{\InqKhL}{\mathbf{InqKhL}}
\newcommand{\inqlor}{\rotatebox[]{90}{$ \eqslantless $}}
\newcommand{\InqDEL}{\mathbf{InqDEL}}

\newcommand{\LInqKh}{\mathbf{DELKh}}
\newcommand{\LEL}{\mathbf{EL}}
\newcommand{\LDEL}{\mathbf{DEL}}

\newcommand{\LPL}{\mathbf{PL}}
\newcommand{\LPP}{\PL^{\Prop}}

\newcommand{\R}{R}
\newcommand{\RR}{\textit{RL}}
\newcommand{\Rel}{\mathcal{R}}

\newcommand{\V}{V}
\renewcommand{\S}{S}
\renewcommand{\SS}{\mathcal{S}}
\renewcommand{\emptyset}{\varnothing}

\newcommand{\SDELKh}{\mathsf{SDELKh}}
\newcommand{\SFive}{\textsf{S5}}
\newcommand{\SFiveL}{\mathbf{S5}}

\newcommand{\SFour}{\textsf{S4}}
\newcommand{\SINTDN}{\mathsf{SINTDN}}

\newcommand{\INTU}{\ensuremath{\mathtt{INTU}}}
\newcommand{\TAUT}{\ensuremath{\mathtt{TAUT}}}
\newcommand{\NECK}{\ensuremath{\mathtt{NEC_\K}}}
\newcommand{\NECB}{\ensuremath{\mathtt{NEC_\Box}}}
\newcommand{\DISTK}{\ensuremath{\mathtt{DIST_\K}}}
\newcommand{\DISTB}{\ensuremath{\mathtt{DIST_\Box}}}
\newcommand{\AxTK}{\ensuremath{\mathtt{T_{\K}}}}
\newcommand{\AxTransK}{\ensuremath{\mathtt{4_{\K}}}}
\newcommand{\AxEucK}{\ensuremath{\mathtt{5_{\K}}}}
\newcommand{\AxTB}{\ensuremath{\mathtt{T_{\Box}}}}
\newcommand{\AxTransB}{\ensuremath{\mathtt{4_{\Box}}}}
\newcommand{\MP}{\ensuremath{\mathtt{MP}}}
\newcommand{\KKhp}{\ensuremath{\mathtt{KKhp}}}
\newcommand{\KhK}{\ensuremath{\mathtt{KhK}}}
\newcommand{\KhC}{\ensuremath{\mathtt{Kh\land}}}
\newcommand{\KhD}{\ensuremath{\mathtt{Kh\lor}}}
\newcommand{\KhI}{\ensuremath{\mathtt{Kh\!\to}}}
\newcommand{\RKhI}{\ensuremath{\mathtt{RKh\!\to}}}
\newcommand{\Per}{\ensuremath{\mathtt{Per}}}
\newcommand{\Ver}{\ensuremath{\mathtt{Ver}}}
\newcommand{\PR}{\ensuremath{\mathtt{PR}}}
\newcommand{\EU}{\ensuremath{\mathtt{EU}}}
\newcommand{\Khbot}{\ensuremath{\mathtt{Kh\bot}}}
\newcommand{\AxTransKh}{\ensuremath{\mathtt{4_\Kh}}}
\newcommand{\AxEucKh}{\ensuremath{\mathtt{5_\Kh}}}
\newcommand{\INV}{\ensuremath{\mathtt{INV}}}
\newcommand{\KINV}{\ensuremath{\mathtt{KINV}}}
\newcommand{\BD}{\ensuremath{\mathtt{B\lor}}}
\newcommand{\hKINV}{\ensuremath{\mathtt{hKINV}}}
\newcommand{\BKD}{\ensuremath{\mathtt{BK\lor}}}

\newcommand{\rRE}{\ensuremath{\mathtt{rRE}}}
\newcommand{\NDkh}{\ensuremath{\mathtt{KhND}}}
\newcommand{\ND}{\ensuremath{\mathtt{ND}}}
\newcommand{\KP}{\ensuremath{\mathtt{KP}}}
\newcommand{\KhKP}{\ensuremath{\mathtt{KhKP}}}

\newcommand{\KhPeirce}{\ensuremath{\mathtt{KhPEIRCEp}}}
\newcommand{\KKhN}{\ensuremath{\mathtt{KKhN}}}

\newcommand{\DNkh}{\ensuremath{\mathtt{KhDNp}}}
\newcommand{\DN}{\ensuremath{\mathtt{DNp}}}

\begin{document}
%\begin{frontmatter}
%\nocite{*}

  \title{Inquisitive Logic as an Epistemic Logic of Knowing How\footnote{An earlier draft of the paper to appear in \textit{Annals of Pure and Applied Logic} with the same title. Please check the published paper for the final version. }}
  \author{Haoyu Wang}
    % \address{Department of Philosophy, Peking University}
  \author{Yanjing Wang}%\footnote{You can put your email address or grant acknowledgement as a footnote, if you wish.}
    \affil{Department of Philosophy, Peking University}
  \author{Yunsong Wang}
    \affil{ILLC, University of Amsterdam}

\maketitle

\begin{abstract}
In this paper, we present an alternative interpretation of propositional inquisitive logic as an epistemic logic of \textit{knowing how}. In our setting, an inquisitive logic formula $\alpha$ being supported by a state is formalized as \textit{knowing how to resolve $\alpha$} (more colloquially, \textit{knowing how $\alpha$ is true}) holds on the S5 epistemic model corresponding to the state. Based on this epistemic interpretation, we use a dynamic epistemic logic with both know-how and know-that operators to capture the epistemic information behind the innocent-looking connectives in inquisitive logic. We show that the set of valid know-how formulas corresponds precisely to the inquisitive logic. The main result is a complete axiomatization with intuitive axioms using the full dynamic epistemic language. Moreover, we show that the know-how operator and the dynamic operator can both be eliminated without changing the expressivity over models, which is consistent with the modal translation of inquisitive logic existing in the literature. We hope our framework can give an intuitive alternative interpretation to various concepts and technical results in inquisitive logic, and also provide a powerful and flexible tool to handle both the inquisitive reasoning and declarative reasoning in an epistemic context.  
\end{abstract}

% \begin{keywords}
% inquisitive logic \sep epistemic logic \sep logic of knowing how \sep intuitionistic logic \sep BHK-interpretation \sep realization \sep resolution
% %% keywords here, in the form: keyword \sep keyword
% \end{keywords}

%\end{frontmatter}

\maketitle

\section{Introduction}
%\noteYW{I rewrote the intro quite a bit. Please double check for typos etc.}
\textit{Inquisitive logic} captures the valid reasoning patterns of statements and questions in a neat uniform framework \cite{Ciardelli2011}. There are two major views on inquisitive logic: one may view it as a non-classical logic as presented in the early days of the field (cf. e.g., \cite{Ciardelli2011}); alternatively, as endorsed by various recent works, one can also view the framework as a conservative extension of classical logic taking the inquisitive disjunction and other machinery as new additions on top of the classical ones (cf. e.g., \cite{Ciardelli16phd,Ciardelli2018}). According to the first \textit{non-classical view}, the basic system $\InqL$ of propositional inquisitive logic is a \textit{weak intermediate logic} that includes all the axioms of intuitionistic logic without the axiom of excluded middle, but is \textit{not} closed under uniform substitution. Moreover, inquisitive logic has some surprising close connections to some other logics such as Medvedev Logic \cite{Ciardelli2011}, and it can be viewed as a disguised \textit{propositional intuitionistic dependence logic} \cite{Yang2014,Ciardelli2020}. The second \textit{extension-view} gives flexibility in designing logical systems combining the power of both classical reasoning and inquisitive reasoning. As a rapidly growing field of logic, besides being intensively studied itself, inquisitive logic has been extended widely with various modalities and other non-classical connectives (cf. e.g., \cite{CiardelliR15,CiardelliB19,Holliday20,PuncocharS21}).

The characteristic semantic feature of inquisitive logic is that it is based on the \textit{support} relation over  \textit{information states} (or simply \textit{states}), instead of the usual satisfaction relation over \textit{possible worlds} widely used in various classical and non-classical logics. As a logic, $\InqL$ collects all the valid propositional formulas supported by  all the states. Intuitively, to a modal logician, a state in (propositional) inquisitive semantics is simply \textit{a set} of possible worlds. To an epistemic logician, a state can be further viewed as an epistemic ($\SFive$) Kripke model, where the agent is not sure which possible world is the actual one.\footnote{It is interesting to note that \textit{information states} in game theory are precisely those epistemically indistinguishable possible worlds.}  It is very natural to ask whether there is an intrinsic connection between epistemic logic and inquisitive logic, given both are defined over similar models. In fact, there is a syntactic translation from inquisitive logic to modal and epistemic logic (e.g., \cite[Section 5.4]{Ciardelli16phd}). In this paper, we focus on whether we can have an intuitive epistemic interpretation semantically. 

Our approach is based on the non-classical view of inquisitive logic. 
%The connection is realised technically and conceptually by giving an alternative epistemic interpretation via an epistemic logic of \textit{knowing how}, which we hope can also give a very intuitive epistemic interpretation of the non-classical operators in inquisitive logic and its semantics. This new interpretation can also give a way to combine the inquisitive reasoning and classical epistemic reasoning, which is close to the heart of the extension view of inquisitive logic. The journey may take some time, and we will explain the basic ideas first in this introduction. 
Thanks to the intimate connection between inquisitive logic and Medvedev logic that we mentioned, we can make use of a crucial observation about the epistemic interpretation of intuitionistic and intermediate logics. Inspired by the original finite-problem semantics of Medvedev logic \cite{Med62}, Wang \cite{Wang2021} proposed to interpret intuitionistic truth of a formula $\alpha$ as \textit{knowing how to prove\slash solve $\alpha$}. Similar informal ideas of understanding intuitionistic truth as knowledge-how appeared a few times in the literature of intuitionistic logic in the past century (cf. e.g., \cite{Melikhov2013}), starting from the very first paper by Heyting explaining the intuitive meaning of intuitionistic statements  \cite{Heyting1930}: 
\begin{quote}
To satisfy the intuitionistic demands, the statement must be the realisation of the expectation expressed by the proposition $p$. Here, then, is the \textit{Brouwerian statement} of $p$: It is known how to prove $p$. 
% We will denote this by $\vdash p$. The words ``to prove'' must be taken in the sense of ``to prove by construction''. [...] $\vdash \neg p $ will mean: ``It is known how to reduce $p$ to a contradiction''.
\end{quote}
In contrast with the previous informal philosophical discussions, we now have a formal way to capture  this interpretation of intuitionistic truth not merely conceptually but also mathematically, based on the techniques for epistemic logics of \textit{know-wh} proposed and studied by Wang (cf. e.g., \cite{Wang2018,Wang17d}). The main idea is to introduce the so-call \textit{bundled modalities}, which pack a quantifier and an epistemic modality together to formalize the \textit{de re} knowledge expressed by \textit{knowing how\slash why\slash what and so on} (cf. e.g., \cite{Wang2016,Xu19}). For example, \textit{knowing how to prove $\alpha$} (written as $\Kh\alpha$) can be rendered as \textit{there exists} a proof $\rho$ such that \textit{it is known that} $\rho$ proves $\alpha$. Bundled modalities also lead to bundled fragments of first-order modal logic, which are often decidable \cite{Wang17d,PadmanabhaRW18,Abfgd22}. 

Combining this technique with some formalized BHK-interpretation of intuitionistic logic may allow us to turn intuitionistic logic and various intermediate logics into epistemic logics of \textit{knowing how}. The general method is to use a powerful epistemic language based on classical logic to \textit{unload} the implicit epistemic content  hidden behind the propositional language of propositional  intuitionistic logic, foreseen by Hintikka and van Benthem viewing intuitionistic logic as an \textit{implicit} epistemic logic  \cite{Hintikka2001IntuitionisticLogic,VanBenthem2018}. An intuitionistic logic formula $\alpha$ is first translated into a know-how formula $\Kh\alpha$ in our setting; then, depending on the structure of $\alpha$ and the BHK-interpretation, we can further ``decode'' $\alpha$ by reducing its complexity within  the logical language, e.g., $\Kh(\alpha\lor\beta)$ can be decomposed to $(\Kh\alpha \lor \Kh\beta)$, where the connectives outside the scope of $\Kh$ are \textit{classical}. %As another example, we may decode $\Kh(\alpha\to\beta)$ as $\K\Box (\Kh\alpha\to\Kh\beta)$, where $\K$ is the normal know-that modality and $\Box$ captures the possible informational updates. $\K\Box (\Kh\alpha\to\Kh\beta)$ reveals that the intuitionistic implication is actually about the knowledge of dependency between knowledge-how of $\alpha$ and $\beta$.
%In the similar way, we can decode $\Kh \neg\neg \alpha$ as $\K\alpha$, where $\K$ is the usual know-that modality. 
This also helps us to understand the distinct role of the negation as the bridge between the classical and the intuitionistic settings. Such an epistemic approach can make intuitionistic logic and its relatives more intuitive, and the existing important technical results become more transparent. 
%\noteYW{To add something about explicit and extension view...}

\medskip

In this paper, we apply such ideas to propositional inquisitive logic as a variant of intuitionistic logic under the non-classical view mentioned above. Note that the \textit{intended} interpretation of inquisitive logic does not make reference to knowledge, thus what we are proposing is an \textit{alternative} epistemic  interpretation. Instead of talking about \textit{proofs} and \textit{solutions} in intuitionistic logic, here we are concerned with \textit{resolutions} of issues raised by inquisitive formulas. Roughly speaking, a formula can be true in various ways, e.g., $\alpha\lor \beta$ can be true because of the truth of $\alpha$ or the truth of $\beta$. We say it is \textit{resolved} if it is settled with a particular way of being true. In a nutshell, we interpret \textit{``$s$ supports $\alpha$''} in inquisitive semantics as \textit{``knowing how to \textit{resolve} $\alpha$''} (or simply \textit{``knowing how $\alpha$ is true''}) over the epistemic model corresponding to $s$. For example, \textit{``$s$ supports $\alpha\lor\beta$''}  becomes \textit{``knowing how $\alpha\lor\beta$ is true''} over the corresponding epistemic model, which intuitively requires either \textit{knowing how $\alpha$ is true} or \textit{knowing how $\beta$ is true}, thus  \textit{resolving} the question raised by $\alpha\lor \beta$. This interpretation is also extended to the entailment relation, i.e., $\alpha$ entails $\beta$ in inquisitive logic is interpreted as \textit{knowing how $\alpha$ is true} entails \textit{knowing how $\beta$ is true}, thus interpreting the reasoning in inquisitive logic as \textit{know-how preserving}. This can be considered as the analog of correspondence of entailment in propositional logic and the standard epistemic logic of knowing that.

\medskip

Actually, the idea of interpreting inquisitive formulas in terms of \textit{knowing how} first  appeared in Ciardelli's master thesis in the early days of inquisitive logic \cite{Ciardelli04}. In \cite{Ciardelli2011}, the authors also made it more precise by using a notion of \textit{realization} (or, say, \textit{resolution}) inspired by the BHK-interpretation of intuitionistic logic, to which we will come back in Section \ref{sec.related work} with detailed discussions. 

Technically, compared to intuitionistic logic, there is a crucial simplification in inquisitive logic that each atomic proposition $p$ has one and only one possible resolution $p$. This is due to the assumption in inquisitive semantics that atomic propositions stand for statements, and  questions are formed only via question-forming operators like the inquisitive disjunction (cf. \cite[Section 2.5.5]{Ciardelli16phd}). Translated into our epistemic interpretation, this assumes that you always \textit{know how} to resolve an atomic formula if you \textit{know that} it is true.  On the other hand, \textit{knowing that $p$ is true} is a presupposition of \textit{knowing how it is true}. Therefore, \textit{knowing how} to resolve the atomic proposition $p$ is equivalent to  \textit{knowing that} $p$ is \textit{resolvable} in our epistemic rendering of inquisitive logic. Note that this cannot be extended to an arbitrary formula $\alpha$, which would trivialize the know-how formulas. Nevertheless, as we will see later, this simplification leads to the characteristic features (and the charm) of inquisitive logic. Based on this assumption about atomic propositions, we will eventually be able to eliminate the \textit{knowing how} operator in our ``epistemicization'' of the inquisitive logic.  All these will become more clear, after establishing our technical framework.  %\noteWHY{Maybe we should identify the `characteristic feature' and `the same logical feature' here more clearly. Otherwise it might be a bit confusing.}

\medskip

Here we summarize our main contributions in this paper:
\begin{itemize}
    \item From an epistemic point of view, we propose a dynamic-epistemic logical framework of \textit{knowing how} to give an alternative interpretation of inquisitive logic. 
    \item We show that the propositional inquisitive logic, as a weak  intermediate logic, is exactly the valid know-how fragment of our logic. 
    \item We obtain a complete axiomatization of the full dynamic epistemic logic with intuitive axioms, which can transparently explain the axioms and results of inquisitive logic from our epistemic perspective. 
    \item We show that both the knowing how modality and the dynamic modality can be eliminated in terms of expressivity, and thus inquisitive logic can be viewed as a fragment of epistemic logic.
\end{itemize}
More conceptually, with our approach, we want to bridge: 
\begin{itemize}
    \item \textit{\textbf{possible worlds and states}} by viewing states as epistemic models to base our semantics on.
    \item \textit{\textbf{classical and non-classical logics}} by using a modal logic based on classical connectives to interpret the non-classical inquisitive logic.
    \item \textit{\textbf{non-modal and modal formulations}} by using the epistemic language of knowing how to reveal the rich dynamic-epistemic information behind the propositional formulas of inquisitive logic. 
    \item \textit{\textbf{\textit{de re} and \textit{de dicto} knowledge}} by connecting the \textit{knowledge-how} with \textit{knowledge-that} relevant to inquisitive reasoning.
\end{itemize}
We hope this epistemic approach can further enhance the power of inquisitive logic to wider applications while keeping its spirit. This may make inquisitive logic more intuitive to an audience who are not familiar with non-classical logics and support-based semantics. Future directions are discussed in Section \ref{sec.conc}. 

%We can use the same symbols but capsule the intuitionistic use while keeping the axioms intuitive. 
\paragraph{Structure of the paper} We recall the language and semantics of propositional inquisitive logic in Section \ref{sec.pre}.  Section \ref{sec.khlogic} introduces the language and semantics of our dynamic epistemic logic of knowing how, and shows that it captures the original inquisitive logic precisely as a fragment. In Section \ref{sec.ax}, we obtain an intuitive axiomatization for the full language and show its completeness by using the reduction axioms. Section \ref{sec.understanding} looks at the core concepts in inquisitive semantics from our intuitive epistemic perspective. In Section \ref{sec.related work}, we discuss the related work,  particularly about inquisitive modal logic. We conclude in Section \ref{sec.conc} with future directions.

\section{Preliminaries: Inquisitive Logic}\label{sec.pre}
We first present some basic definitions and results of (propositional) inquisitive logic $\InqL$ following the expositions of  \cite{Ciardelli2011,Ciardelli16phd}.

\begin{definition}[Language $\PL$]
Given a countable set $\Prop$ of proposition letters, the language of propositional logic ($\LPP$) is defined as follows:
$$\alpha::= p\mid \bot\mid (\alpha\land\alpha)\mid (\alpha\lor\alpha)\mid (\alpha\to\alpha)$$
where $p\in \Prop$. For abbreviations, we write {$\neg\alpha$} for $\alpha\to\bot$, $\top$ for $\neg\bot$, and $\alpha\lra\beta$ for $(\alpha\to \beta)\land (\beta\to \alpha)$.
%{$!\alpha$} for $\neg\neg\alpha$, and {$?\alpha$} for $\alpha\lor\neg\alpha$.  \noteYW{Need to check whether we need these abbreviations of $?$ and $!$.} 
\end{definition}
$\Prop$ is assumed to be countable as in \cite{Ciardelli16phd}. 
%Since we have restriction invariance (Proposition 2.8 in \cite{Ciardelli2011}), from now on we only consider the case that $\Prop$ is countably infinite. 
We write $\PL$ for $\LPP$ when $\Prop$ is given in the context. For readability, in the rest of the paper, we often omit the parentheses if no ambiguity arises.   

Instead of the satisfaction relation based on  possible worlds, $\InqL$ adopts the following \textit{support relation} in its semantics based on states, which makes $\InqL$ behave non-classically.

While the \emph{states} were initially defined via the concept of \emph{index} \cite{Ciardelli09}, it was later defined as a subset of the world set of a \textit{(propositional) information model} in the recent literature (cf. e.g., \cite{Ciardelli16phd}), here we adopt the latter definition.\footnote{Note that in contrast with \cite{Ciardelli16phd}, we do not allow $W$ to be empty in order to ease the later presentation without changing any technical result.}

% \noteWHY{do we have to define information and epistemic model separately?} \noteYW{We don't make the distinction then. I will take care of that.  }
\begin{definition}[Information model]
\label{imodel}
Given $\Prop$, an (information) model is a pair $\M=\langle W, \V\rangle$ where: 
 \begin{itemize}
     \item $W$ is a non-empty set of possible worlds; 
     \item $\V: W\to \wp(\Prop)$ is a valuation function.
 \end{itemize}
 Given $\M$, we refer to its components by $W_\M$ and $\V_\M$. A \emph{full model} is a model such that $\V[W]=\wp(\Prop).$
\end{definition}
Following \cite{Ciardelli2011}, the inquisitive semantics of $\LPL$ over information models is given by the \textit{support} relation. 
\begin{definition}[Support]
\label{support}
Given $\Prop$ and an information model $\M=\la W,\V\ra$, an \emph{(information) state} $s\subseteq W$ is a subset of $W$. $W$ and $\emptyset$ are called  \emph{trivial state} and \emph{inconsistent state} respectively. \emph{Support} is a relation between states and formulas (written as $\M,s\Vdash\alpha$):
	\begin{enumerate}
		\item $\M,s\Vdash p$ iff $\forall w\in s, p\in w$.
		\item $\M,s\Vdash\bot$ iff $s=\varnothing$. 
		\item $\M,s\Vdash (\alpha\land\beta)$ iff $\M,s\Vdash\alpha$ and $\M,s\Vdash\beta$.	
		\item $\M,s\Vdash (\alpha\lor\beta)$ iff $\M,s\Vdash\alpha$ or $\M,s\Vdash\beta$.
		\item $\M,s\Vdash (\alpha\to\beta)$ iff $\forall t\subseteq s:$ if $\M,t\Vdash\alpha$ then $\M,t\Vdash\beta$.
	\end{enumerate}
\end{definition}

We write $s\Vdash\alpha$ for $\M,s\Vdash\alpha$ when no confusion arises.

\begin{definition}[Entailment ($\Vdash$)]
	A set of $\PL$-formulas $\Gamma$ entails a $\PL$-formula
$\alpha$ in inquisitive semantics, $\Gamma\Vdash\alpha$, if and only if for any state $s$ in any model $\M$ if $\M,s\Vdash\Gamma$ then $\M,s\Vdash \alpha$. We say $\alpha$ is \emph{valid} if $\Vdash \alpha.$
\end{definition}

\begin{definition}[Inquisitive logic]
	Inquisitive logic, $\InqL$, is the set of $\PL$-formulas that are valid in inquisitive semantics, i.e., the set of formulas that are supported by all states in all models.
\end{definition}
%\noteYW{I changed the definition above and the proposition below to incorporate the model. Please double check. }

It is straightforward to show: 
\begin{proposition} [\cite{Ciardelli2011}]\label{prop.full}
For any $\PL$-formula $\alpha$, 
$\alpha \in \InqL$ iff for any full model $\M$, $\M,W_\M\Vdash \alpha$.
\end{proposition}

Now we present the proof system $\SINTDN$ and  two axioms schemata $\KP$ and $\ND_k$, which in combination give rise to two distinct axiomatizations of $\InqL$.
%Let $\SINTDN$ be the following proof system:
\begin{center}
System $\SINTDN$\\
	\begin{tabular}{ll}
		{\textbf{Axioms}}&\\
		\begin{tabular}[t]{lll}
			$\INTU$ & \text{Intuitionistic validities}\\
			$\DN$& $\neg\neg p\to p$ \text{for all} $p\in\Prop$ \\
 %      \KP &	$(\neg\alpha\to\beta\lor\gamma)\to(\neg\alpha\to\beta)\lor(\neg\alpha\to\gamma)$
			\end{tabular}
	\end{tabular}
	\begin{tabular}{lclclc}
		\textbf{Rules:}\\
	\MP & $\dfrac{\alpha,\alpha\to\beta}{\beta}$\\
		\end{tabular}
	\end{center}
Let $\KP$ and $\ND_k$ be the following axiom schemata:
	\begin{center}
	\begin{tabular}{ll}
       $\KP$ &	$(\neg\alpha\to\beta\lor\gamma)\to(\neg\alpha\to\beta)\lor(\neg\alpha\to\gamma)$\\
       $\ND_k$ &  $(\neg \alpha \to \bigvee_{1\leq i\leq k}
\neg \beta_i)\to \bigvee_{1\leq i\leq k}(\neg \alpha \to \neg \beta_i)$
 		\end{tabular}
	\end{center}

\begin{theorem}[Axiomatizations of $\InqL$ \cite{Ciardelli2011}] 
\label{thm.axiom}
$\SINTDN+\KP$ and $\SINTDN+\{\ND_k \mid \text{ $k\in \mathbb{N}$}\}$ are both sound and complete for $\InqL$.
\end{theorem}
%\noteYW{Do we need the negative translation?}

%\noteYW{I adopted the new notions w.r.t. models. }

The following definitions are taken from \cite{Ciardelli2011} adapted with a given model $\M$ as in \cite{Ciardelli16phd}.\footnote{Following the new notion in \cite{Ciardelli16phd}, we call  \textit{possibilities} (Definition 2.9 in \cite{Ciardelli2011}) as \textit{alternatives}, and call \textit{assertions}  (Definition 2.14 in \cite{Ciardelli2011}) as \textit{statements}.
Note that in Definition 2.14 in \cite{Ciardelli2011}, questions and statements (assertions) are defined absolutely with respect to a full model with the trivial state. Here we generalized it to a relative notion as in the cases of inquisitiveness and informativeness. } 
%Note that although assertions and statements are defined in different ways, they are technically the same. 
These characterize some important concepts in inquisitive semantics. Note that in the earlier literature (c.f.\cite{Ciardelli2011}),  the definition of proposition was the set of alternatives. But the latest definition in \cite{draftCiardelli} has been modified as the downward closure of the old one. We adopt the latest definition.

\begin{definition}[Alternatives  and propositions]
 Let $\alpha$ be a $\LPL$ formula and let $\M$ be a model. 
 \begin{itemize}
 \label{def.possibility}
\item An \emph{alternative} for $\alpha$ in $\M$ is a maximal state $s$ in $\M$ supporting $\alpha$;
\item The \emph{proposition} expressed by $\alpha$ in $\M$ (call it $[\alpha]_\M$), is the set of states in $\M$ that supports $\alpha$. That is, $[\alpha]_\M=\{s\in W_{\M}\mid s\vDash\alpha\}$.
%\item The \emph{truth set} of $\alpha$ in $s$ (call it $|\alpha|_s$), is the set of indices where $\alpha$ is classically true. \noteYW{Will this play a role?}
 \end{itemize}
% \noteWHY{the definition of proposition is different now.}
\end{definition} 
%Given $\M$, if $s$ is the trivial state $W_\M$, then a possibility for $\alpha$ in $s$ is also called an \textit{alternative} for $\alpha$ in $\M$ \cite{Ciardelli16phd}. 
%\noteWHY{I omitted 'truth' in the title}
% \begin{proposition}[Support and possibilities] For any state $s$ and any formula $\alpha\in\LPL$: $s\Vdash\alpha\iff$ $s$ is contained in a possibility for $\alpha$.
% \label{prop.sup.pos}
% \end{proposition}
%\noteYW{To adopt the notion of alternatives and the new notion of propositions without assuming the maximality. In the recent literature (e.g., Ivano's thesis), alternatives are defined w.r.t. models, shall we adopt that? If yes, we need to change the later part (Section 5.) }
\begin{definition}[Inquisitiveness and informativeness]\, \label{def.inqandinfo}
Let $\alpha$ be a $\LPL$ formula and let $\M$ be a  model. 
 \begin{itemize}
     \item $\alpha$ is \emph{inquisitive} in $\M$ if $[\alpha]_\M$ contains at least two alternatives;
\item $\alpha$ is \emph{informative} in $\M$ if there is some world in $\M$ that is not included in
any alternative for $\alpha$ in $\M$.
 \end{itemize}
\end{definition}

%\noteWHY{I took the following definitions here from Section 5.}

Furthermore, we can define the relative notions of \textit{statements} and \textit{questions}.%\footnote{As mentioned, the term ``assertion'' has been replaced by ``statement'' in the latest literature. We adopt the latest terminology here. Note that in Definition 2.14 in \cite{Ciardelli2011}, questions and statements (assertions) are defined absolutely with respect to a full model with the trivial state. Here we generalized it to a relative notion as in the cases of inquisitiveness and informativeness. } 
\begin{definition}[Questions and statements]\,
Given a model $\M$: 
\begin{itemize}
 \item $\alpha$ is a \emph{question} in $\M$ iff it is not informative in $\M$;
\item $\alpha$ is a \emph{statement} in $\M$ iff it is not inquisitive in $\M$.
\end{itemize}
\end{definition}

We will come back to these in Section \ref{sec.understanding} from our epistemic perspective. Before that, we shall introduce our framework of the logic of knowing how to interpret inquisitive logic epistemically.

\section{Inquisitive logic as a logic of knowing how}\label{sec.khlogic}
In this section, we give a formal epistemic interpretation of inquisitive logic by using an epistemic logic of knowing how. 

\subsection{Language and models}

We first introduce our dynamic epistemic language of knowing how, where on top of $\PL$ we add three modalities of know-that, know-how, and updates.  

\begin{definition}[Language $\LInqKh$]
Given a countable set of proposition letters $\Prop$, the \emph{Dynamic Epistemic Language of Knowing How} ($\LInqKh^\Prop$) is defined as follows:

$$\varphi::= p\mid \bot\mid (\varphi\land\varphi)\mid (\varphi\lor\varphi)\mid (\varphi\to\varphi)\mid \K\varphi\mid \Kh\alpha \mid \Box\varphi $$
where  $p\in \Prop$ and $\alpha\in\LPP$. For abbreviations, we write {$\hK$} for $\neg\K\neg$, {$\hKh$} for $\neg\Kh\neg$ and write {$\Diamond$} for $\neg\Box\neg$. We denote the $\Kh$-free fragment as $\LDEL^\Prop$, and the $\Box$-free fragment of $\LDEL^\Prop$ as $\LEL^\Prop$.
\end{definition}
Again, we often omit the $\Prop$ from $\LInqKh^\Prop$ when $\Prop$ is fixed in the context. Intuitively, $\K \phi$ expresses that ``the agent \textit{knows that} $\phi$'', $\Kh\alpha$ says that ``the agent \textit{knows how} to resolve $\alpha$'' or simply ``the agent \textit{knows how}  $\alpha$ is true'',
%\footnote{For example, $\Kh(p\lor q)$ says that knowing how $p\lor q$ is true, which requires knowing which of the two atomic propositions is true. } 
and $\Box\phi$ says that ``$\phi$ holds, no matter what further information is given''. Note that $\Kh$ only takes $\PL$-formulas $\alpha$ whereas $\K$ and $\Box$ can be combined with any  $\LInqKh$-formulas $\phi$. Therefore we can express $\K\neg \Kh\alpha$ but not $\Kh\K\alpha$ in $\LInqKh$. 

\medskip

% $\LInqKh$ is interpreted on epistemic models with resolutions for basic propositions in $\Prop$. 
% \begin{definition}[Models]
% \label{model}
% Given $\Prop$, a model of $\LInqKh$ is a tuple $\M=\langle W,\sim\,, \V\rangle$ where: 
%  \begin{itemize}
%      \item $W$ is a non-empty set of possible worlds; 
%      \item $\sim\ \subseteq W\times W$ is an equivalence relation;
%      %a binary relation such that for any $w,v\in W$, $w\sim v$, i.e. $\sim\,= W\times W$;
%      \item $\R: W\times\Prop\to \wp(\Prop)$ is a function such that for any $w\in W$, $\R(w,p)\subseteq\{p\}$.
%  \end{itemize}
%  Given $\M$, we refer to its components by $W_\M$, $\sim_\M$ and $\R_\M$. 
% Given a model $\M$ and $w\in \M$, $\M,w$ is a \textsl{pointed model}. A \emph{connected model} $\M$ is a model where $\sim\ =W_\M\times W_\M$. 
% \end{definition}
% By the definition of $\R$, for any $p\in\Prop$ and any $w\in \M$, $\R(w, p)$ is either $\emptyset$ or $\{p\}$, which reflects the underlying assumption in inquisitive semantics that atomic propositions do not bring inquisitiveness themselves. In other words, we always \textit{know how} to resolve atomic propositions when possible. This is \textit{the} most fundamental deference between $\InqL$ and other intermediate logics such as Medvedev logic.   It will become more clear when we discuss the axioms later.  \noteYW{To add citations later on.}

$\LInqKh$ will be interpreted on standard \textit{single-agent epistemic models} where the (implicit) epistemic relation is the \textit{total relation}. Technically speaking, \textbf{such  models are exactly the information models} as we defined in Definition \ref{imodel}, when we omit the epistemic relation since it is always  total.\footnote{Nevertheless, in a multi-agent setting, explicit epistemic relations are necessary. We leave it for a future occasion. } For this reason, in the sequel of the paper, we will also call information models \textit{epistemic models}.  

For \textbf{notational convenience}, we also write $w\in \M$ in case that $w\in W_{\M}$, $\M'\subseteq \M$ in case that $\M'$ is a submodel of $\M$. If $w\in \M'\subseteq \M$ then we write $(\M',w)\subseteq (\M,w)$. 

Although the models are simply \SFive\ epistemic models, in order to reflect the non-classical features of $\InqL$, we define the semantics for $\Kh$ via the resolutions, based on the idea that knowing how $\alpha$ is true means knowing a particular resolution for $\alpha$, in line with the BHK-interpretation of the intuitionistic connectives. We first define the resolution space for each formula below in Definition \ref{def.rs}. The actual resolutions on each world for each formula, to be defined in Definition \ref{def.resolution}, will be subsets of the resolution space. 
%\noteW{maybe there are some syntax errors}
\begin{definition}[Resolution space]
\label{def.rs}
$\S$ is a function assigning each $\alpha$ its set of potential resolutions:
$$\begin{aligned}
\S(p)&=\{p\},\text{ for }p\in\Prop\\
\S(\bot)&=\{\bot\}\\
\S(\alpha\lor\beta)&=(\S(\alpha)\times\{0\})\cup(\S(\beta)\times\{1\})\\
\S(\alpha\land\beta)&=\S(\alpha)\times\S(\beta)\\
\S(\alpha\to\beta)&=\S(\beta)^{\S(\alpha)}\end{aligned}$$
%where  $Y^X$ is the set of all functions from $X$ into $Y$. 
Let $\SS=\bigcup_{\alpha\in\LPL}\S(\alpha)$.
\end{definition}
Our definition is similar to the definition of realizations in \cite[Part 4]{miglioli1989}, which was proposed not in the context of inquisitive logic. The definition of resolution space reflects the intuition that for each $\alpha$, there is a set of possible ways to \textit{resolve it as an issue} or say to \textit{make it true}. Intuitively, the set of potential resolutions of a disjunction is the disjoint union of the potential resolutions of each disjunct: to make a disjunction true you need to explicitly make one of the disjuncts true. The resolution space of a conjunction is the Cartesian product of the resolution space of each conjunct: to make a conjunction true, you need to make both conjuncts true. The resolution space for an implication, is the set of functions from the resolution of the antecedent to the resolution space of the consequent: to make an implication true, you need to have a way to transform any resolution of the antecedent to some resolution of the consequent.

%\noteW{truth maker is undefined}

Note that the potential resolutions for atomic propositions are assumed to be singletons, which reflects the underlying assumption in inquisitive semantics that atomic propositions do not bring inquisitiveness themselves. In other words, we always \textit{know how} to resolve atomic propositions when possible. This is \textit{the} most fundamental difference between $\InqL$ and other intermediate logics such as Medvedev logic. It will become more clear when we discuss the axioms later. Technically speaking, the possible resolution of each atomic proposition $p$ is not necessary $p$ itself, as long as it is unique, and it will become more clear when the semantics is introduced. 

For technical convenience, in the line of the resolution space for atomic propositions, the resolution space of $\bot$ is defined as $\{\bot\}$, but as we will see later $\bot$ does not have any real resolution.

%Since for any $w$ and $\alpha\in\LPL$, $\S(w,\alpha)$ is independent from points in $W$, we can obtain that for any $w,v\in W$ and for any $p\in\Prop$, $\S(w,p)=\S(v,p)$. ($\ast$). From now on we simply write $\S(\alpha)$ for potential resolutions of $\alpha$. 
%We further define the set $\S$ of potential resolutions of all formulas in $\LPL$. Formally,  $\S=\bigcup_{\alpha\in\LPL}\S(\alpha)$.

It is obvious that the resolution space for each $\alpha\in \PL$ are non-empty and finite :
\begin{proposition}\label{prop.non-emptyS}
For any $\alpha\in$ $\LPL$, $\S(\alpha)\neq\varnothing$ and it is finite.
\end{proposition}
Now given a model, we can generate the \textit{(actual)  resolutions} of each $\alpha$ on each world.  
\begin{definition}[Resolution]
\label{def.resolution}
Given a model $\M$, the resolution function $\R: W\times\LPL\to \SS$ is defined as follows: 
 $$\begin{aligned}
\R(w,\bot)&=\varnothing\\
\R(w,p)&=\left\{\begin{array}{cl}
  \{p\}   & \text{ if $p\in V_\M(w)$}  \\
 \emptyset   & \text{ otherwise}
\end{array}\right.\\
\R(w,\alpha\lor\beta)&=(\R(w,\alpha)\times\{0\})\cup(\R(w,\beta)\times\{1\})\\
\R(w,\alpha\land\beta)&=\R(w,\alpha)\times\R(w,\beta)\\
%\R(w,\alpha\land\beta)&=\{\langle x,y\rangle\mid x\in\R(w,\alpha), y\in\R(w,\alpha)\}\\
\R(w,\alpha\to\beta)&=
\{f\in\S(\beta)^{\S(\alpha)}\mid f[\R(w,\alpha)]\subseteq\R(w,\beta)\}\end{aligned}$$
For $U\subseteq W$, we write $\R(U, \alpha)$ for $\bigcap_{w\in U}\R(w, \alpha).$ When $U=W_\M$ we also write $\R(\M, \alpha)$ for $\R(U, \alpha).$
\end{definition}

We call $\R(w,\alpha)$ the set of resolutions for an issue $\alpha$ on a given world $w$. It is clear that $R(w,\alpha)\subseteq \S(\alpha).$

Note that although the resolution space of $\bot$ is non-empty for technical convenience, it cannot have any actual resolution on any world. An atomic proposition $p$ can have the  resolution $p$ only when it is true on $w$. The resolutions for an implication $\alpha\to \beta$ on a possible world are the functions in the resolution space of $\alpha\to \beta$ mapping a resolution of $\alpha$ to a  resolution of $\beta$ on the same world in line with the BHK-interpretation.  

Recall that $\neg \alpha$ is the abbreviation of $\alpha\to\bot$. Since negation plays an important role in intermediate logics, we have the following observation, which is useful for later discussions. 
\begin{proposition}
\label{prop.negation} For any $\M, w$, any $\alpha$,  
$\R(w, \neg \alpha)$ is either $\emptyset$ or a fixed singleton set independent from $w$, and $\R(w, \neg \alpha)=\emptyset$ iff $R(w, \alpha)\not=\emptyset$.
\end{proposition}
\begin{proof} 
By the definitions of $\R$ and $\S$:
\begin{align*}
&\R(w, \alpha\to\bot)\\
=&\{f\in\S(\bot)^{\S(\alpha)}\mid f[\R(w, \alpha)]\subseteq\R(w,\bot)\}\\
=&\{f\in\{\bot\}^{\S(\alpha)}\mid f[\R(w, \alpha)]\subseteq\emptyset\}\\
=&\{f\in \{\bot\}^{\S(\alpha)}\mid f[\R(w, \alpha)]=\emptyset\}\\
=& \left\{\begin{array}{cl}
  \emptyset   & \text{ if $\R(w, \alpha)\not=\emptyset$}  \\
 \{f^\alpha_\bot\}   & \text{ if $\R(w, \alpha)=\emptyset$} 
 \end{array}\right.
\end{align*}
where $f^\alpha_\bot$ is the constant function such that $f^\alpha_\bot(x)=\bot$ for any $x\in \S(\alpha)$. Note that $f^\alpha_\bot$ only depends on $\alpha$ and it is independent from $w$.
\end{proof}

It is also important that $\R(w, \alpha)$ only depends on the valuation on $w$ itself but not on other worlds. 
\begin{proposition}\label{prop.rind}
For any $\M,w$ and $\N,v$, if $V_\M(w)=V_\N(v)$ then $\R(w, \alpha)=\R(v,\alpha)$ for all $\alpha\in\PL$. 
\end{proposition}

% \begin{proposition}[All resolutions are potential resolutions]
% For any $\alpha\in$ $\LPL$, and any pointed model $\M,w$, $\R(w, \alpha)\subseteq \S(\alpha)$.
% \end{proposition}

%\begin{definition}[Restriction of model and representative model]\label{restriction}
%Let $\Prop$ be a countably infinite set of propositional letters. Let $\mathcal{P'}\subseteq\Prop$. Given a model $\M=\langle W,\sim\,,\V\rangle$. For any $w\in W$, Let $w\!\!\upharpoonright_{\mathcal{P'}}=V(w)\cap\mathcal{P'}$. 

 %By (AC) there is a function $f:W_{\mathcal{P'}}\to W$, s.t. $f([w]\!\!\upharpoonright_{\mathcal{P'}})\in[w]\!\!\upharpoonright_{\mathcal{P'}}$ for any $w\in W$. 

%Then for any model $\M$, the restriction of $\M$ to $\mathcal{P'}$, i.e. $\M\!\!\upharpoonright_{\mathcal{P'}}=\langle W_{\mathcal{P'}},\sim_{\mathcal{P'}},\R_{\mathcal{P'}}\rangle$ is a $\mathcal{P'}$-model s.t.  $W_{\mathcal{P'}}=\{w\!\!\upharpoonright_{\mathcal{P'}}\mid w\in W\}$, $w\!\!\upharpoonright_{\mathcal{P'}}\sim_{\mathcal{P'}} v\!\!\upharpoonright_{\mathcal{P'}}$ for any $w,v\in W_{\mathcal{P'}}$ and $\R_{\mathcal{P'}}(w\!\!\upharpoonright_{\mathcal{P'}},p)=\R(w,p)$ for any  $p\in\mathcal{P'}$.

%For any model $\M$ we call $\M\!\!\upharpoonright_{\Prop}$ the \textsl{representative model} of $\M$.
%良定义需要证明吗？
%\end{definition}

%Note that since for any $p\in\mathcal{P'}$, and any $w\in W$, $\R(w,p)=\{p\}$ or $\R(w,p)=\varnothing$, when $\mathcal{P'}$ is finite, say, of the size $n$, $[w]\!\!\upharpoonright_{\mathcal{P'}}$ is also finite, of a size not greater than $2^n$.

\subsection{Semantics}
%old semantics with $\sim$
% Given the definitions of the epistemic models and resolutions, we can define the semantics of $\LInqKh$ over epistemic models.
% \begin{definition}[Semantics]
%  For any $\varphi\in\LInqKh$ and a pointed model $\M,w$ where $\M=\langle W,\sim\,,\V\rangle$, the satisfaction relation is defined as below
% $$\begin{array}{|lcl|}
% \hline
% \M,w\nvDash\bot &&  \\
% %\M,w\vDash p &\iff \R(w,p)=\{p\}\\
% \M,w\vDash p &\iff& p\in \V(w)\\
% \M,w\vDash (\varphi\lor\psi)&\iff& \M,w\vDash \varphi \text{ or }\M,w\vDash \psi\\
% \M,w\vDash (\varphi\land\psi)&\iff& \M,w\vDash \varphi \text{ and }\M,w\vDash \psi\\
% \M,w\vDash (\varphi\to\psi)&\iff& \M,w\vDash \varphi \text{ implies } \M,w\vDash \psi\\
% \M,w\vDash \Box\varphi&\iff& \text{ for any } \M',w\subseteq \M,w, \M',w\vDash\varphi\\
% \M,w\vDash \K\varphi&\iff& \text{ for any } v\in \M,  \M,v\vDash\varphi\\
% \M,w\vDash \Kh\alpha&\iff& \text{ there exists an } x\in \S(\alpha) \text{ s.t. for any } v\sim w, x\in\R(v,\alpha)\\
% \hline
% \end{array}$$
% We say a set of $\LInqKh$-formula $\Gamma$ \emph{entails} another formula
% $\varphi$ $(\Gamma\vDash\varphi)$, if for any pointed model $\M,w$, $\M,w\vDash\Gamma$ implies $\M,w\vDash\varphi$. We say $\varphi$ is \emph{valid} ($\vDash\phi$) if $\emptyset \vDash \varphi.$ A formula schema is valid if all its instances are valid. 
% \end{definition}
% Note that the semantics of $\Kh$ can be reformulated as below, which we will often use for notational brevity.
% $$\begin{array}{|lcl|}
% \hline
% \M,w\vDash \Kh\alpha&\iff& \R(W,\alpha)\not=\emptyset\\
% \hline
% \end{array}$$

Given the definition of resolutions, we can define the satisfaction relation of $\LInqKh$ on possible worlds in epistemic models. Note that the connectives outside the scope of $\Kh$ are classical and the semantics of $\K$ is standard as in epistemic logic. The semantics of $\Kh\alpha$ on a world $w$ is the formalization of the idea that the agent knows how to resolve $\alpha$ iff it knows a particular resolution of $\alpha$. The semantics of the dynamic modality  $\Box$ is based on taking submodels as in a version of arbitrary announcement logic \cite{balbiani2008}, representing informational updates in terms of eliminating possibilities. 
\begin{definition}[Semantics]
 For any $\varphi\in\LInqKh$ and a pointed model $\M,w$ where $\M=\langle W,\V\rangle$, the satisfaction relation is defined as below:
$$\begin{array}{|lcl|}
\hline
\M,w\nvDash\bot &&  \\
%\M,w\vDash p &\iff \R(w,p)=\{p\}\\
\M,w\vDash p &\iff& p\in \V(w)\\
\M,w\vDash (\varphi\lor\psi)&\iff& \M,w\vDash \varphi \text{ or }\M,w\vDash \psi\\
\M,w\vDash (\varphi\land\psi)&\iff& \M,w\vDash \varphi \text{ and }\M,w\vDash \psi\\
\M,w\vDash (\varphi\to\psi)&\iff& \M,w\vDash \varphi \text{ implies } \M,w\vDash \psi\\
\M,w\vDash \Box\varphi&\iff& \text{ for any } (\M',w)\subseteq (\M, w), \M',w\vDash\varphi\\
\M,w\vDash \K\varphi&\iff& \text{ for any } v\in \M,  \M,v\vDash\varphi\\
\M,w\vDash \Kh\alpha&\iff& \text{ there exists a } x\in \S(\alpha) \text{ s.t. for any } v\in \M, x\in\R(v,\alpha)\\
\hline
\end{array}$$
We say a formula is \emph{valid on $\M$} ($\M\vDash \phi$) if $\M,w\vDash\phi$ for all $w\in \M$. 
We say a set of $\LInqKh$-formula $\Gamma$ \emph{entails} another formula
$\varphi$ $(\Gamma\vDash\varphi)$, if for any pointed model $\M,w$, $\M,w\vDash\Gamma$ implies $\M,w\vDash\varphi$. We say $\varphi$ is \emph{valid} ($\vDash\phi$) if $\emptyset \vDash \varphi.$ A formula schema is valid if all its instances are valid. 
\end{definition}
It is not hard to see that the structure of the truth condition of $\Kh$ is in terms of the bundle $\exists x \K$ as in other know-wh logics \cite{Wang2018,Wang17d}.  Recall that $\R(\M, \alpha)=(\bigcap_{v\in\M}\R(v,\alpha))$ (cf.\ Definition \ref{def.resolution}). The semantics of $\Kh$ can be then  reformulated as below for notational brevity.
$$\begin{array}{|lcl|}
\hline
\M,w\vDash \Kh\alpha&\iff& \R(\M, \alpha)\not=\emptyset\\
\hline
\end{array}$$

%Given a pointed model $\M,w$ and a set of formulas $\Phi\in\LInqKh$, for notational convenience, we write $\M,w\vDash \Phi$ iff $\M,w\vDash\varphi$ for $\varphi\in \Phi$. We write $\varphi\vDash\psi$ iff for any pointed model $\M,w$, $\M,w\vDash\varphi\implies \M,w\vDash \psi$.  We write $ \M\vDash\varphi$ iff $\M,w\vDash\varphi$ for each $w\in \M$.

% \begin{definition}[Entailment and validity]
% A set of formulas $\Gamma$ entails a formula
% $\varphi$ in inquisitive knowhow semantics, $\Gamma\vDash\varphi$, if and only if for any pointed model $\M,w$, if $\M,w\vDash\psi$ for each $\psi\in\Phi$, then $\M,w\vDash\varphi$.
% \end{definition} 

% \begin{definition}[Validity]
%  Given any $\varphi\in\LInqKh$, $\varphi$ is valid in Inquisitive-Knowhow semantics iff for any pointed model $\M,w$, $\M,w\vDash \varphi$.
% \end{definition} 

% \begin{definition}[$\InqKhL$]
% Inquisitive Logic of Knowing How, $\InqKhL$, is the set of valid formulas in Inquisitive-Knowhow semantics.
% \end{definition}

% It follows immediately that
% %for any $\alpha\in\LPL$ and any pointed model $\M,w$, $\M,w\vDash\alpha$ or $\M,w\vDash\neg\alpha$. And 
% if there are $w,w'\in M$ s.t. $V(w)=V(w')$ where $V$ is defined as in Definition \ref{V}, then for any $\varphi\in\LInqKh$, $\M,w\vDash\varphi$ iff $\M,w'\vDash\varphi$. This result can be established by a straightforward induction on $\varphi$.

Note that the truth conditions of $\K$ and $\Kh$ do not depend on the designated world, therefore we have: 
\begin{proposition}
\label{prop.KHK} For any model $\M,w$: 
\begin{itemize}
    \item $\M,w\vDash\Kh\alpha\iff \M\vDash \Kh\alpha$,  and $\M,w\vDash\neg \Kh\alpha\iff \M\vDash \neg \Kh\alpha$;
    \item $\M,w\vDash \K\alpha\iff \M\vDash \K\alpha$, and  $\M,w\vDash\neg \K\alpha\iff \M\vDash \neg\K\alpha$.
\end{itemize}
    As a consequence, the introspection axioms $\Kh\alpha\lra \K\Kh\alpha$ and $\neg \Kh\alpha\lra \K \neg\Kh \alpha$ are valid.
% For any $\alpha\in\LPL$ and any $\M,w$ we have:
% $$\M,w\vDash\Kh\alpha\iff \M,w\vDash\K\Kh\alpha\iff \M\vDash\Kh\alpha$$
\end{proposition}

It is clear that the above semantics is simply classical for $\alpha\in \PL$. In particular, $\{\alpha\in \PL\mid\  \vDash \alpha\}$ is classical propositional logic ($\CPL$).  However, let $\InqKhL$ be $\{\alpha\mid \ \vDash\Kh\alpha\}\subseteq \PL$, we will show that $\InqKhL=\InqL$. Before that, to understand the semantics of $\Kh$ better, we first show that the classical semantics of $\alpha$ can be viewed from the perspective of resolutions: \textit{resolvability equals truth}. 
%Proposition \ref{prop.nept} provides another view to understand Proposition 2.21 in \cite{Ciardelli2011}. 
Note that for atomic formulas $p$, $\M,w\vDash p \iff R(w,p)\not=\emptyset$ is obviously true, however we need to generalize it to any formula $\alpha\in\LPL$. 
\begin{proposition}
\label{prop.nept}
For any $\alpha\in\LPL$ and pointed model $\M,w$. $\M,w\vDash \alpha \iff \R(w,\alpha)\neq\varnothing, \text{ where }\alpha\in\LPL$.
\end{proposition}

\begin{proof}
Induction on the structure of $\alpha$: 
$$\begin{aligned}
\M,w\vDash p &\iff p\in V_\M(w) \iff \R(w,p)=\{p\}\iff\R(w,p)\neq\varnothing \\
\M,w\not\vDash\bot&\iff\R(w,\bot)=\varnothing\\
\M,w\vDash ( \alpha\lor\beta) &\iff \M,w\vDash \alpha \text{ or } \M,w\vDash \beta \iff \R(w,\alpha)\neq\varnothing \text{ or }\R(w,\beta)\neq\varnothing\\&\iff \text{ there exists an } x\in\R(w,\alpha) \text{ or there exists a } y\in\R(w,\beta)\\&\iff\text{ there exists a pair } \langle x,0\rangle \text{ or }\langle y,1\rangle \text{ in } \R(w,\alpha\lor\beta)\\&\iff\R(w,\alpha\lor\beta)\neq\varnothing\\
\M,w\vDash (\alpha\land\beta) &\iff \M,w\vDash \alpha \text{ and } \M,w\vDash \beta \iff \R(w,\alpha)\neq\varnothing \text{ and }\R(w,\beta)\neq\varnothing\\&\iff \text{ there exists an } x\in\R(w,\alpha) \text{ and there exists a } y\in\R(w,\beta)\\&\iff \text{ there exists a pair } \langle x,y\rangle \in\R(w,\alpha\land\beta)\\&\iff\R(w,\alpha\land\beta)\neq\varnothing\\
\M,w\vDash (\alpha\to\beta) &\iff \M,w\vDash \alpha \text{ implies } \M,w\vDash \beta \iff \R(w,\alpha)\neq\varnothing \text{ implies } \R(w,\beta)\neq\varnothing\\&\iff\S(\beta)^{\S(\alpha)}\neq\varnothing\text{ and } f[\R(w,\alpha)]\subseteq\R(w,\beta) \text{ is possible}\\&\iff\R(w,\alpha\to\beta)\neq\varnothing
\end{aligned}$$
\end{proof}
For negations, based on Propositions \ref{prop.nept} and \ref{prop.negation}, we have $$\M,w \vDash \neg \alpha \iff \R(w,\alpha)=\emptyset\iff \M,w \nvDash \alpha.$$

Now, based on Proposition \ref{prop.nept}, we have an alternative semantics for $\K\alpha$ for $\alpha \in \LPL$:
 $$\begin{array}{|lcl|}
\hline
\M,w\vDash \K\alpha&\iff& \textbf{ for any } v\in \M, \text{ \textbf{there exists} an } x\in\R(v,\alpha)\\
\hline
\end{array}$$
Compared with the truth condition of $\Kh$, it now becomes clear that the distinction between $\Kh$ and $\K$ is exactly the distinction between the  \textit{de re} and \textit{de dicto} knowledge, i.e., knowing $\alpha$ is resolvable vs. knowing how $\alpha$ is resolved. 
$$\begin{array}{|lcl|}
\hline
\M,w\vDash \Kh\alpha&\iff& \text{ \textbf{there exists} an } x \text{ s.t. \textbf{for any} } v\in \M, x\in\R(v,\alpha)\\
\hline
\end{array}$$
Based on this distinction, $\Kh\alpha$ is clearly stronger than $\K \alpha$: 
\begin{proposition}\label{prop.Kh2K}
 $\Kh \alpha\to \K \alpha$ is valid for all $\alpha\in\PL$. 
\end{proposition}
The distinction disappears if we consider the atomic propositions since there can be at most one fixed resolution for each $p\in \PL$. 
\begin{proposition}\label{prop.K2Kh}
 $\Kh p\lra \K p$ is valid for all $p\in\Prop$. 
\end{proposition}

However, $ \K \alpha\to \Kh \alpha$ does not hold in general. 
\begin{example}\label{ex.unisub}
$\Kh (p\lor \neg p)\lra \K (p\lor \neg p)$ is not valid, e.g., in the model 
$\lr{\{w,v\}, V}$ where $p\in V(w)$ but $p\not\in V(v)$, $\K(p\lor\neg p)$ holds everywhere but $\Kh(p\lor\neg p)$ holds nowhere. 
%$$\xymatrix{w: \{p\} \ar@{-}[r]& v: \emptyset }$$
\end{example}

The example also shows that the valid $\Kh$-formulas are not closed under uniform substitution, to which we will come back in Section \ref{sec.ax}. 

It is well-known that in intermediate logics negation plays a role in bridging the classical and the intuitionitic validities, we show that this can be understood by the fact that $\neg$ bridges $\Kh$ and $\K$ in our setting. 
\begin{proposition}[Negative translation]\label{prop.negtrans}
$\Kh\neg \alpha \lra \K \neg \alpha$ is valid. As a consequence, $\Kh\neg \neg\alpha \lra \K \alpha$ is valid. 
\end{proposition}
\begin{proof}
By (the proof of) Proposition \ref{prop.negation} $$\R(v,\neg \alpha)\not=\emptyset \iff \R(v,\neg \alpha)=\{f^\alpha_\bot\} \iff \R(v,\alpha)=\emptyset $$ 
Thus $\M, w\vDash \Kh \neg \alpha \iff \R(\M,\neg \alpha)\not=\emptyset \iff \R(v, \alpha)=\emptyset \text{ for all } v\in\M \iff \M,w \vDash \K\neg \alpha$. Therefore $\M, w\vDash \Kh \neg \neg \alpha \iff \M,w \vDash \K\neg \neg \alpha \iff \M,w \vDash \K \alpha.$ 
\end{proof}

Although we cannot reduce $\Kh$ to $\K$ in general, we can reduce the complexity of $\alpha$ in $\Kh\alpha$ step by step, which will play an important role in the later sections. 
 \begin{proposition}\label{prop.khreduction}
The following formulas and schemata are valid: 
	\begin{itemize}
		\item $\Kh \bot\lra \bot$
%		\item $\Kh\neg \alpha \lra \K\neg \alpha$
		\item $\Kh(\alpha\lor\beta)\leftrightarrow \Kh\alpha\lor\Kh\beta$
		\item $\Kh(\alpha\land\beta)\leftrightarrow \Kh\alpha\land\Kh\beta$
		\item $\Kh(\alpha\to\beta)\lra \K\Box(\Kh\alpha\to\Kh\beta)$
%		\item $\Kh(p\lor\neg p)$
%		\item[$\mathsf{KP}$] $(\neg\alpha\to\beta\lor\gamma)\to(\neg\alpha\to\beta)\lor(\neg\alpha\to\gamma)$
	\end{itemize}
\end{proposition}
%The failure of uniform substitution (within the $\Kh$ operator) also shows up 
	\begin{proof} We only show the non-trivial cases of $\Kh(\alpha\lor\beta)$ and $\Kh(\alpha\to\beta)$. 
	%For any pointed model $\M,w$ of $\LInqKh$. 
\begin{itemize}

% 	\item $\M,w\vDash\Kh(p)$
% 	$\iff \text{ there exists an } x \text{ s.t. for any } v\sim w, x\in\R(v,p)\\\iff \R(v,p)=\{p\} \text{ for each } v\in M\\\iff p\in v \text { for each } v\in M\\\iff \M,v\vDash p \text { for each } v\in M\\\iff \text{ s.t. for any } w\sim v, \M,v\vDash p\\\iff \M,w\vDash\Kp$
% 	\item $\M,w\vDash\Kh(\bot)$
% 	$\\\iff\text{there exists an } x\in\varnothing\text{ (which is not possible)}\\\iff \M,w\vDash \bot$
    \item $\M,w\vDash\Kh(\alpha\lor\beta)\iff \R(\M,\alpha\lor \beta)\not=\emptyset$\\ $\iff \text{ there exists an } \lr{x, 0} \in \R(\M,\alpha\lor \beta) \text{ or there exists a } \lr{y, 1} \in \R(\M,\alpha\lor \beta)$\\ $\iff \text{ there exists an } x\in \R(\M,\alpha) \text { or there exists a } y\in \R(\M,\beta)$\\$\iff \M,w\vDash \Kh\alpha\text{ or } \M,w\vDash\Kh\beta\iff \M,w\vDash \Kh\alpha\lor\Kh\beta$
%     \item $\M,w\vDash\Kh(\alpha\lor\beta)$\\
%   $\iff \text{ there exists an } x\text{ s.t. for any } v\sim w, x\in\R(v,\alpha\lor\beta) \\\iff \text{ there exists an } \lr{a, 0}\text{ s.t. for any } v\sim w, x\in \{\langle x,1\rangle \mid x\in\R(v,\alpha)\}\ (\text{let it be } a)\text { or } \\\qquad\ \text{ there exists an } x\text{ s.t. for any } v\sim w, x\in\{\langle x,2\rangle \mid x\in\R(v,\beta)\}\ (\text{let it be } b)\\\iff \text{ there exists an } x\text{ s.t. for any } v\sim w, x\in\R(v,\alpha) \text { or }\\\qquad\ \text{ there exists an } x\text{ s.t. for any } v\sim w,  x\in\R(v,\beta)\\\iff \M,w\vDash \Kh\alpha\text{ or } \M,w\vDash\Kh\beta\\\iff \M,w\vDash \Kh\alpha\lor\Kh\beta$
%  \item  $\M,w\vDash\Kh(\alpha\land\beta)$
% $\iff \text{ there exists an } x\text{ s.t. for any } v\sim w, x\in\R(v,\alpha\land\beta)\\\iff \text{ there exists an } x\text{ s.t. for any } v\sim w, x\in \{\langle x,y\rangle \mid x\in\R(v,\alpha)\text{ and }x\in\R(v,\beta)\}\ (\text{let it be } \langle a,b\rangle ) \\\iff \text{ there exists an } x\text{ s.t. for any } v\sim w, x\in\R(v,\alpha)\ (\text{let it be } a) \text { and }\\\qquad\ \text{ there exists an } x\text{ s.t. for any } v\sim w,  x\in\R(v,\beta)\ (\text{let it be } b)\\\iff \M,w\vDash \Kh\alpha\text{ and } \M,w\vDash\Kh\beta\\\iff \M,w\vDash \Kh\alpha\land\Kh\beta$
 \item Let us now consider the case for $\Kh(\alpha\to\beta)\lra \K\Box(\Kh\alpha\to\Kh\beta)$.\\
 $\Longrightarrow$: Suppose $\M,w\vDash \Kh(\alpha\to\beta)$, then by the semantics, there is some $f\in \R(\M,\alpha\to\beta).$
 Towards a contradiction, suppose $\M,w\not\vDash\K\Box(\Kh\alpha\to\Kh\beta)$.
 That is, there is an $v\in \M$ and an $\M',v\subseteq \M,v$ s.t. $\M',v\vDash\Kh\alpha$ but $\M',v\not\vDash\Kh\beta$. So there is an $x\in \R(\M',\alpha)$. Recall that $f$ is a function with domain $\S(\alpha)$, and $\S(\alpha) \supseteq\R(u,\alpha)$ for all $u\in\M'$, thus $x\in Dom(f)$. Moreover, since $f\in \R(\M,\alpha\to\beta)$, $f\in \R(\M',\alpha\to\beta).$ Let $y=f(x)$. By the definition of $\R(\M',\alpha\to\beta)$, $y\in\R(u,\beta)$ for each $u\in \M'$. Therefore $\M',v\vDash\Kh\beta$, a contradiction.
 
%  By Proposition \ref{prop.KHK},  $\M,w\vDash \Kh(\alpha\to\beta)\iff M\vDash \Kh(\alpha\to\beta)\iff$ for any $v\sim w\in M$, $\M,v\vDash \Kh(\alpha\to\beta) $. By Proposition \ref{persistence}, $\M,v\vDash \Box\Kh(\alpha\to\beta)$, i.e. for any $\M',v\subseteq \M,v$, there is an $x$ s.t. for any $u\sim v\in M'$, $x\in\R(u,\alpha\to\beta)$ ($\ast$).\\
%  Towards a contradiction, suppose $\M,w\not\vDash\K\Box(\Kh\alpha\to\Kh\beta)$.
%  That is, there is an $v\in M$ and an $\M',v\subseteq \M,v$ s.t. $v\sim w$, $\M',v\vDash\Kh\alpha$ and $\M',v\not\vDash\Kh\beta$. So there is a $y$ s.t. for any $u\sim v\in M'$, $y\in\R(u,\alpha)$. Since $x$ in ($\ast$) is a function with domain $\S(\alpha)\supseteq\R(u,\alpha)$, $y\in Dom(x)$. Let $z=x(y)$. By definition of $x$, $z\in\R(u,\beta)$ for each $u$ s.t. $u\sim v\in M'$. Therefore $\M',v\not\vDash\Kh\beta$, a contradiction.
 $\Longleftarrow$: Suppose $\M,w\vDash\K\Box(\Kh\alpha\to\Kh\beta)$, then for all $v\in\M$, $\M,v\vDash\Box(\Kh\alpha\to\Kh\beta)$. By the semantics of $\Box$, for any $v\in \M$ and for any $\M',v\subseteq \M,v$, $\M',v\vDash\Kh\alpha\to\Kh\beta$ ($\ast$). By Proposition \ref{prop.non-emptyS}, $\S(\alpha)$ is finite and non-empty, thus we can assume $\S(\alpha)=\{x_0,x_1,\dots,x_n\}$ for some $n\in\mathbb{N}$. For $i\in\{0,\dots,n\}$, let $W_i=\{w\mid x_i\in\R(w,\alpha)\}$. If $W_i$ is not empty then let $\M_i$ be the submodel of $\M$ such that $W_{\M_i}=W_i$. Clearly $x_i\in\R(W_i,\alpha)$, therefore for any $u\in \M_i$, $\M_i,u\vDash\Kh\alpha$. By ($\ast$) we have $\M_i,u\vDash\Kh\beta$ thus there is a $y_i\in\R(W_i,\beta)$. Now fix a $y\in \S(\beta)\not=\emptyset$, let $f=\{\langle x_i,y_i\rangle\mid i\in\{0,\dots,n\} \text{ and $W_i\not=\emptyset$}\}\cup \{\langle x_i,y\rangle\mid i\in\{0,\dots,n\} \text{ and $W_i=\emptyset$}\}$. Clearly $f\in\S(\beta)^{\S(\alpha)}$. Now for any $v\in \M$ and $i\in\{0,\dots,n\}$, if $x_i\in\R(w,\alpha)$ \text{ then } $v\in W_i$ by the definition of $W_i$, thus  $y_i\in\R(v,\beta)$ by the construction of $f$. Therefore $f[\R(v,\alpha)]\subseteq \R(v,\beta)$ for all $v\in \M$. It follows that $\M,v\vDash \Kh(\alpha\to\beta)$ for all $v\in \M$ including $w$. Note that the axiom of choice is not needed in the above finitary constructions.
\end{itemize}
\end{proof}
\begin{remark}
Propositions \ref{prop.K2Kh} and \ref{prop.khreduction} will help us to  eliminate the $\Kh$ modalities without changing the expressive power. Actually, we can also eliminate the $\Box$ modality eventually. We will discuss the reduction formally in Section \ref{sec.ax} when discussing the axiomatization featuring the corresponding reduction axioms. 
\end{remark}
Now we have an intuitive reading of $\alpha\lor\neg\alpha$ in inquisitive logic based on  Propositions \ref{prop.khreduction} and \ref{prop.negtrans}.
$$\M,w \vDash \Kh(\alpha\lor\neg\alpha)\iff \M,w \vDash \Kh\alpha\lor\Kh\neg\alpha\iff \M,w \vDash \Kh\alpha\lor\K\neg\alpha$$
The formula $\Kh\alpha\lor\K\neg\alpha$ says either you know how to resolve $\alpha$ or you know it is not resolvable\slash true, and it is clearly not valid in general. This explains intuitively why inquisitive logic does not accept the law of excluded middle. 

Given the validity of $\Kh(\alpha\to\beta)\lra \K\Box(\Kh\alpha\to\Kh\beta)$ and Proposition \ref{prop.KHK}, we can give an alternative compositional truth condition to $\Kh(\alpha\to\beta)$, which is more handy to use. 
\begin{proposition}\label{prop.alternativeimp}
For any model $\M$, $\M,w\vDash \Kh(\alpha\to\beta)$ iff for any $\M'\subseteq \M$, $\M\vDash\Kh\alpha$ implies $\M\vDash \Kh\beta$. 
\end{proposition}
\begin{proof} It suffices to show that $\M,w\vDash \K\Box(\Kh\alpha\to\Kh\beta)$ iff for any $\M'\subseteq \M$, $\M\vDash\Kh\alpha$ implies $\M\vDash \Kh\beta$, which follows directly from the semantics of $\K\Box$. Note that  Given a pointed model $\M,w$, $\Box$ quantifies over all the submodels of $\M'$ such that $w\in \M'$. Since on an $S5$ model, $\K$ refers to all the points $w\in\M$, the combination of $\K\Box$ quantifies over all the submodels of $\M$. That is, for any $\M'\subseteq\M$ and $v\in\M'$, $\M',v\vDash\Kh\alpha$ implies $\M',v\vDash\Kh\beta$. By Proposition \ref{prop.KHK} the proof is completed.
\end{proof}
%\noteYW{Provide details}
%\noteW{It can also be emphasized that the implication is classical implication rather than intuitionistic implication} \noteYW{The meta language is classical, as assumed usually. }

In the following, given $\Gamma\subseteq \LPL,$ let $\Kh\Gamma=\{\Kh\gamma \mid \gamma\in\Gamma\}$. As another consequence of the validity of $\Kh(\alpha\to\beta)\lra \K\Box(\Kh\alpha\to\Kh\beta)$, we have: 
\begin{proposition}\label{prop.validkh}
$\Kh\Gamma\vDash \Kh(\alpha\to \beta)$ iff $\Kh\Gamma\vDash \Kh\alpha\to \Kh\beta$. As a special case, $\vDash \Kh(\alpha\to \beta)$ iff $\vDash \Kh\alpha\to \Kh\beta$.
\end{proposition}
\begin{proof}
It suffices to show that $\Kh\Gamma\vDash \K\Box(\Kh\alpha\to\Kh\beta)$ iff $\Kh\Gamma\vDash \Kh\alpha\to \Kh\beta$
$\Longrightarrow$ is based on the fact that $\K\phi\to\phi$ and $\Box\phi\to\phi$ are valid. $\Longleftarrow$ is based on the fact that if $\M, w\vDash \Kh\Gamma$ then any for submodel $\M'$ of $\M$, $\M'\vDash \Kh\Gamma$ by the semantics of $\Kh$. Indeed, if for all the worlds in $\M$, there exists a uniform resolution for each formula in $\Gamma$, the same resolutions will certainly serve as the uniform resolutions for worlds in $\M'$. Therefore, supposing $\Kh\Gamma\vDash \Kh\alpha\to \Kh\beta$ and $\M,w\vDash\Kh\Gamma$, it follows that $\Kh\alpha\to \Kh\beta$ is satisfied on all the submodels of $\M$. By Proposition \ref{prop.KHK}, that is, for any $\M'\subseteq \M$, $\M\vDash\Kh\alpha$ implies $\M\vDash \Kh\beta$. By Proposition \ref{prop.alternativeimp}, we have $\M,w\vDash\Kh(\alpha\to \beta).$
\end{proof}
%\noteYW{Provide details}
%\noteW{the $\NECK$ and $\NECB$}\noteYW{the rules are for validities...}

Based on Proposition \ref{prop.validkh}, we have the following theorem for $\Kh$ formulas, which is the counterpart of Proposition 3.10 (deduction theorem) in \cite{Ciardelli2011}. 
%\noteWHY{It seems that the following theorem can be proved without using the proposition above. In fact the version that corresponds to the deduction theorem of inquisitive logic is $\Gamma'\cup\Kh\Gamma\vDash\Kh\alpha\iff\Gamma'\vDash\Kh\bigwedge_{\beta\in\Gamma}\beta\to\Kh\alpha$, the proof of which is commented as written below.}
\begin{theorem}\label{thm.deductionthm}
% For any $\alpha, \beta\in\LPL$ and $\Gamma\in\LInqKh$, $\Gamma \cup\{ \Kh\alpha\}\vDash \Kh\beta$ iff $\Gamma\vDash \Kh\alpha\to \Kh\beta.$
For any $\alpha\in\LPL$ and $\Gamma, \Gamma' \subseteq\LPL$ such that $\Gamma$ is finite, $$\Kh(\Gamma'\cup\Gamma)\vDash\Kh\alpha\iff\Kh\Gamma'\vDash\Kh(\bigwedge_{\gamma\in\Gamma}\gamma\to \alpha).$$ 
%\noteWHY{I deleted $\beta$ and replaced 'and' for 'such that'. I wonder whether we can prove a more generalized version where $\Gamma'$ is a subset of $\LInqKh$}
\end{theorem}
\begin{proof}
% $$\begin{aligned}
% \Gamma \cup\{ \Kh\alpha\}\vDash \Kh\beta &\iff \text{ for any } \M,w \text{ such that }\M,w\vDash\Gamma \cup\{ \Kh\alpha\}, \text{ then }\M,w\vDash\Kh\beta\\&\iff \text{ for any } \M,w \text{ such that }\M,w\vDash\Gamma, \text{ if }\M,w\vDash\Kh\alpha, \text{ then }\M,w\vDash\Kh\beta\\&\iff \text{ for any } \M,w \text{ such that }\M,w\vDash\Gamma, \M,w\vDash\Kh\alpha\to\Kh\beta\\&\iff\Gamma\vDash\Kh\alpha\to\Kh\beta\end{aligned}$$
$$\begin{aligned}
\Kh(\Gamma'\cup\Gamma)\vDash\Kh\alpha 
&\iff \text{ for any } \M,w \text{ such that }\M,w\vDash\Kh\Gamma'\cup\Kh\Gamma, \text{ then }\M,w\vDash\Kh\alpha\\
&\iff \text{ for any } \M,w \text{ such that }\M,w\vDash\Kh\Gamma', \text{ if }\M,w\vDash\bigwedge_{\gamma\in\Gamma}\Kh\gamma, \text{ then }\M,w\vDash\Kh\alpha\\
&\iff \text{ for any } \M,w \text{ such that }\M,w\vDash\Kh\Gamma', \M,w\vDash\bigwedge_{\gamma\in\Gamma}\Kh\gamma\to\Kh\alpha\\
&\iff\Kh\Gamma'\vDash\bigwedge_{\gamma\in\Gamma}\Kh\gamma\to\Kh\alpha\\
&\iff\Kh\Gamma'\vDash\Kh\bigwedge_{\gamma\in\Gamma}\gamma\to\Kh\alpha \qquad (\text{by Proposition \ref{prop.khreduction}})\\
&\iff\Kh\Gamma'\vDash\Kh(\bigwedge_{\gamma\in\Gamma}\gamma\to\alpha) \qquad (\text{by Proposition \ref{prop.validkh}})
\end{aligned}$$
\end{proof}

As an analog of the \textit{persistence} of inquisitive formulas over sub-states \cite[Prop. 2.4]{Ciardelli2011}, we show the persistence of $\Kh\alpha$ in our setting with $\Box$. Intuitively, once we know how to resolve $\alpha$, we will not forget, even given more information. 
\begin{proposition}[Persistence]
\label{persistence}
$\Kh\alpha\lra\Box\Kh\alpha$ is valid for any $\alpha\in\LPL$.
\end{proposition}
\begin{proof}
$\Longrightarrow$: Suppose $\M,w \vDash\Kh \alpha$ then $\R(\M,\alpha)\not=\emptyset$. It is clear that $\R(\M,\alpha)\subseteq \R(\M',\alpha)$ for any submodel $\M'$. Therefore $\R(\M',\alpha)\not=\emptyset$ for any submodel $\M'$, thus $\M,w \vDash  \Box\Kh\alpha$. 

$\Longleftarrow$: It is trivial since a model is also a submodel of itself. 
\end{proof}

Another important interaction property between $\Kh$ and $\Box$ is the following, which can be compared to Proposition 2.5 in \cite{Ciardelli2011} regarding the singleton state. 
%\noteWHY{, which is valid due to the existence of a singleton submodel on each pointed model}: 
\begin{proposition}\label{prop.vervalidy}
$ \alpha\lra\Diamond\Kh\alpha$ is valid for any $\alpha\in\PL$. 
\end{proposition}
\begin{proof}
$\Longrightarrow:$ It is valid because we can always go to a singleton submodel containing the current world only, where $\K\alpha$ holds. Note that for singleton models, i.e., models with only one world, $\K\alpha\to\Kh\alpha$ holds trivially for any $\alpha$ since any actual resolution on that single world will be the uniform resolution.

$\Longleftarrow$ is based on the fact that the updates do not change the truth values of propositional formulas on the current world, and the validity of $\Kh\alpha\to\K\alpha$ and $\K\alpha\to\alpha$.
\end{proof}
%In Section \ref{sec.ax}, we will show how to eliminate $\Kh$ and $\Box$ in the language qua expressivity, based on our axioms to be introduced. 

As a feature distinguishing intuitionistic logic and classical logic, disjunction property is also an important property of inquisitive logic \cite[Prop. 3.9 ]{Ciardelli2011}. It holds naturally in our logic.
\begin{proposition}[Disjunction Property] For any formulas $\alpha,\beta\in\LPL$, $\vDash\Kh(\alpha \lor \beta)   \iff  \vDash\Kh\alpha\text{ or }\vDash\Kh\beta.$ 
\label{prop.disjp}
\end{proposition}
% \begin{proposition}[Disjunction Property] For any set of formulas $\Gamma\in\LInqKh$ and $\alpha, \beta\in\LPL$, $\Gamma\vDash\Kh(\alpha \lor \beta)   \iff  (\Gamma\vDash\Kh\alpha\text{ or } \Gamma\vDash\Kh\beta)$ \noteYW{Is it really true? Do we need to restrict the $\Gamma$?}
% \end{proposition}
\begin{proof} 
$\Longrightarrow$: Suppose $\nvDash\Kh\alpha$ and $\nvDash \Kh\beta$, then we have some models $\M,w$ and $\N,v$ such that 
$\M,w\nvDash \Kh\alpha$ and $\N,v \nvDash \Kh\beta$. Now we can simply merge the two models together as the disjoint union $\M\uplus\N$.
%=\{W,\V\}$ such that $W=W_\M\uplus W_\N$, $\V(x,0)=\V_{\M}(x)$ for $x\in\M$ and $\V(x,1)=\V_{\N}(x)$ for $x\in\N$, where $\uplus$ is the disjoint union operation. 
By the semantics of $\Kh$ it is clear that $\M\uplus\N, w\nvDash \Kh\alpha\lor\Kh\beta.$ By Proposition \ref{prop.khreduction}, $\nvDash \Kh(\alpha\lor\beta).$ 

$\Longleftarrow$ is trivial by Proposition \ref{prop.khreduction}.
\end{proof}

As some examples, we show the validity of $\Kh$-versions of some axioms and valid formulas in $\InqL$ (cf. Theorem $\ref{thm.axiom}$). 

\begin{proposition}\label{prop.axvalidity}
The following are valid:
\begin{center}
\begin{tabular}{ll}
%$\CRMK$ & $\Box\Diamond\alpha\leftrightarrow \Diamond\Box\alpha$ \\
$\DNkh$ &$\Kh(\neg\neg p\to p)$  \text{ for $p\in\Prop$}\\
$\KhPeirce$ &$\Kh(((p\to q)\to p)\to p)$ \text{ for $p, q\in\Prop$}\\
$\KhKP$ & $\Kh( (\neg\alpha\to\beta\lor\gamma)\to(\neg\alpha\to\beta)\lor(\neg\alpha\to\gamma))$\\
$\NDkh_k$ & $\Kh((\neg \alpha \to \bigvee_{1\leq i\leq k}
\neg \beta_i)\to \bigvee_{1\leq i\leq k}(\neg \alpha \to \neg \beta_i))$
\end{tabular}
\end{center}
\end{proposition}
%\noteW{we need to emphasize that the first two axioms are atomic version}
\begin{proof}
For $\DNkh$: by Proposition \ref{prop.validkh}, it suffices to check $\vDash\Kh\neg\neg p\to \Kh p.$ By Proposition \ref{prop.negtrans}, it amounts to check $\vDash\K p\to \Kh p,$ which is valid by Proposition \ref{prop.K2Kh}. 

For $\KhPeirce$: by Proposition \ref{prop.validkh}, it suffices to check $\vDash\Kh((p\to q)\to p)\to \Kh p$ which amounts to  $\vDash\K\Box(\K\Box (\Kh p\to \Kh q)\to \Kh p)\to \Kh p$. By Proposition \ref{prop.K2Kh}, we just need to check  $\vDash\K\Box(\K\Box (\K p\to \K q)\to \K p)\to \K p.$ Now suppose $\M,w\nvDash \K p$ then there is $v\in \M$ such that $\M,v\vDash \neg p.$ We need to show that $\M,w\nvDash \K\Box(\K\Box (\K p\to \K q)\to \K p)$. Take the singleton submodel $\M'$ of $\M$ with the world $v$ only, then it is clear that $\M',v\nvDash \K\Box (\K p\to \K q)\to \K p.$ Therefore $\M,w\nvDash \K\Box(\K\Box (\K p\to \K q)\to \K p)$. 

% For $\NDkh_k$: By Proposition \ref{prop.validkh}, it suffices to check $\vDash \Kh(\neg \alpha \to \bigvee_{1\leq i\leq k}
% \neg \beta_i)\to \Kh\bigvee_{1\leq i\leq k}(\neg \alpha \to \neg \beta_i).$ Due to Proposition \ref{prop.khreduction}, we just need to check $\vDash \K\Box (\Kh\neg \alpha \to \Kh\bigvee_{1\leq i\leq k}
% \neg \beta_i)\to \bigvee_{1\leq i\leq k}\Kh(\neg \alpha \to \neg \beta_i)$ which amounts to $\vDash \K\Box (\K\neg \alpha \to \bigvee_{1\leq i\leq k}
% \K\neg \beta_i)\to \bigvee_{1\leq i\leq k}\K\Box(\K\neg \alpha \to \K\neg \beta_i)$ by Proposition \ref{prop.negtrans}. Now suppose $\M,w\nvDash \bigvee_{1\leq i\leq k}\K\Box(\K\neg \alpha \to \K\neg \beta_i),$ thus for each $i$ there is a submodel $\M_i$ of $\M$ such that $\M_i\vDash \K\neg\alpha \land \hK\beta_i.$ Now the union of $\M_i$ is a countermodel for $\K\neg \alpha \to \bigvee_{1\leq i\leq k}
% \K\neg \beta_i.$ Since the union of $\M_i$ is still a submodel of $\M$, we have $\M,w\nvDash \K\Box (\K\neg \alpha \to \bigvee_{1\leq i\leq k}
% \K\neg \beta_i). $ 

For $\KhKP$: 
By Proposition \ref{prop.validkh}, we only need to check $\vDash\Kh(\neg\alpha\to\beta\lor\gamma)\to\Kh((\neg\alpha\to\beta)\lor(\neg\alpha\to\gamma))$. 
It amounts to  $\vDash\K\Box(\K\neg\alpha\to(\Kh\beta\lor\Kh\gamma))\to\K\Box(\K\neg\alpha\to\Kh\beta)\lor\K\Box(\K\neg\alpha\to\Kh\gamma)$ based on Proposition~\ref{prop.khreduction}. We prove its contraposition. 
Suppose $\M,w \vDash\neg \K\Box(\K\neg\alpha\to\Kh\beta)\land \neg\K\Box(\K\neg\alpha\to\Kh\gamma)$ then there are two submodels $\M'_1$ and $\M'_2$ such that $\M'_1\vDash\K\neg\alpha\land\neg\Kh\beta$ and $\M'_2\vDash \K\neg\alpha\land \neg\Kh\gamma.$ Then the union $\M'_1\cup\M'_2$ is a submodel of $\M$ making $\K\neg\alpha\to(\Kh\beta\lor\Kh\gamma)$ false. Thus $\M,w\nvDash\K\Box(\K\neg\alpha\to(\Kh\beta\lor\Kh\gamma))$.
% If $\M,w\vDash\Kh(\neg\alpha\to\beta\lor\gamma)$, then there exists an $f\in\R(\M,\neg\alpha\to\beta\lor\gamma)$. If $\R(\M,\neg\alpha)=\varnothing$, then it naturally follows that  $\M,w\vDash\Kh(\neg\alpha\to\beta)\lor(\neg\alpha\to\gamma)$. Otherwise $\R(\M,\neg\alpha)$ is a singleton, which follows from Proposition \ref{prop.negation}. Then $f[\R(\M,\neg\alpha)]$ is also a singleton. Since $f[\R(\M,\neg\alpha)]\subseteq \R(\M,\beta\lor\gamma)=\R(\M,\beta)\cup\R(\M,\gamma)$, $f[\R(\M,\neg\alpha)]\subseteq \R(\M,\beta)$ or $f[\R(\M,\neg\alpha)]\subseteq \R(\M,\gamma)$. Let $f'\in \S(\beta\lor\gamma)^{\S(\alpha)}$ s.t. $Ran(f')=f[\R(\M,\neg\alpha)]$. Clearly $f'\in \R(\M,\neg\alpha\to \beta)$ or $f'\in \R(\M,\neg\alpha\to \gamma)$. Therefore $\R(\M,\neg\alpha\to \beta\lor\gamma)$ is not empty, i.e. $\M,w\vDash\Kh(\neg\alpha\to \beta\lor\gamma)$.

The validity of $\NDkh_k$ can be proved similarly as in the case of $\KhKP$.
\end{proof}

%While cases of $\KP$, $\ND_k$ and the atomic double negation formulas serve as axioms of $\InqL$ as presented in Theorem $\ref{thm.axiom}$, the atomic cases of Pierce's law are a set of validities in $\InqKhL$ which are not valid in intuitionistic logic. 
Note that Peirce's schema $\Kh(((\alpha\to \beta)\to \alpha)\to \alpha)$ is not valid in general. For example, take the instance where $\alpha=p\lor\neg p$ and $\beta=p$ and the full model is a counterexample.

So far, we have seen the $\alpha$s in the valid $\Kh\alpha$ formulas behave pretty much like the valid formulas in the inquisitive logic, we will show it is no coincidence. 

\subsection{$\InqKhL=\InqL$}
In this subsection, we show $\InqKhL=\{\alpha\in \PL \mid\ \vDash \Kh\alpha\}$ is exactly the inquisitive logic $\InqL$. We will actually prove a stronger result showing the corresponding semantics consequences are the same. %Due to Proposition \ref{prop.disM}, we only consider distinguishing models. It is easy to see that given $\Prop$, a state $s$ can be viewed as a set of valuations, thus the (non-empty) states and distinguishing models are essentially the same. 
\begin{definition}
    %Given a (distinguishing) model $\M$, let $s_{\M}$ be the state  $V_\M[W_\M]=\{V_{\M}(w)\mid w\in \M\}$. 
    Given any model $\M=\la W,\V \ra$ and a non-empty state $s\subseteq W$, let $\M_s$ be the submodel $\langle s, \V|_{s}
    %\V\rightharpoonup_{s}
    \rangle$ where 
    %$\V\rightharpoonup_{s}
    $\V|_{s}$ is the restriction of $\V$ on $s$, that is, $\V(w)=
    %\V\rightharpoonup_{s}
    \V|_{s}(w)$ for any $w\in s$. 
\end{definition}
Here is an easy observation that the worlds outside the given state are irrelevant according to support semantics. This also justifies our notion of $s\Vdash\alpha$ without specifying the model $\M$ below.
%\noteYW{I added a proposition here to be used by Proposition 33 below.}
\begin{proposition} \label{prop.irre}
Given any two models $\M$ and $\M'$, and a state $s\subseteq W_\M$ and $s\subseteq W_{\M'}$, if $\M_s=\M'_s$, then for any $\alpha\in \LPL$: $\M, s\Vdash\alpha\iff \M', s\Vdash \alpha$. In particular if $s$ is non-empty and $s\subseteq W_\M$, then $\M, s\Vdash\alpha\iff \M_s, s\Vdash \alpha$.
\end{proposition}

Now we establish the logical equivalence between $\M$ and $(\M,W_\M)$ where $W_\M$ is viewed as the trivial state.
%\noteYW{Check the proofs below.}
\begin{lemma}
\label{lem.modelstate}
    Given any $\alpha\in\LPL$ and any pointed %(distinguishing)
    model $\M,w$, $$\M,w\vDash\Kh\alpha \iff \M\vDash\Kh\alpha \iff \M, W_\M\Vdash\alpha.$$
\end{lemma}
%这个名字可以吗？题目在说语义，定理用逻辑的形式给出

\begin{proof} By Proposition \ref{prop.KHK}, we only need to show that for any model $\M$,  $\M\vDash\Kh\alpha$ iff $W_\M\Vdash\alpha$. 
%By Proposition \ref{prop.disM}, we can assume $\M$ is a distinguishing model. 
We prove the lemma by induction on $\alpha$. 
$$\begin{aligned}
\M\vDash\Kh p&
\iff \M\vDash \K p\iff \text{ for each } w\in W_\M, p\in \V(w)\iff  W_\M\Vdash p.\\
\M\not\vDash\Kh\bot&
\iff W_\M\neq\varnothing\iff W_\M\not\Vdash\bot.\\
\M\vDash\Kh (\alpha\lor\beta)&\iff \M\vDash\Kh \alpha \lor\Kh \beta \iff \M\vDash\Kh \alpha \text{ or }\M\vDash\Kh \beta\iff W_\M\Vdash\alpha\text{ or } W_\M\Vdash\beta\\&\iff W_\M\Vdash\alpha\lor\beta.\\
\M\vDash\Kh (\alpha\land\beta)&
%\iff \text{ there exists an } \langle x,y\rangle \text{ s.t. for any }v\sim w, \langle x,i\rangle\in\R(w,\alpha\lor\beta)\\&\iff \text{ there exists an } x \text{ s.t. for any }v\sim w, x\in\R(v,\alpha)\\&\text{ and there exists a } y\text{ s.t. for any }v\sim w, y\in\R(v,\beta)\\&
\iff \M\vDash\Kh \alpha \land\Kh \beta \iff \M\vDash\Kh \alpha \text{ and }\M\vDash\Kh \beta\iff W_\M\Vdash\alpha\text{ and } W_\M\Vdash\beta\\& \iff W_\M\Vdash\alpha\land\beta.\\
\M\vDash \Kh(\alpha\to\beta)&\iff \text{ for any } \M'\subseteq \M, \M'\vDash\Kh \alpha \text{ implies }\M'\vDash\Kh \beta \text{ (by Proposition \ref{prop.alternativeimp})}\\
%&\iff\text{ for any } s_{\M'}\subseteq W_\M, s_{\M'}\Vdash\alpha\text{ implies }s_{\M'}\Vdash\beta \text{ (by IH)}\\
&\iff\text{ for any non-empty } s\subseteq W_\M, s\Vdash\alpha\text{ implies }s\Vdash\beta \text{ (by IH)} \\&\iff W_\M\Vdash\alpha\to\beta.
\end{aligned}$$
%The last step is due to the fact that $W_\M\Vdash\alpha\to\beta$ iff for all \textit{non-empty} $s\subseteq W_\M,$ $s\Vdash\alpha$ implies $s\Vdash\beta.$ 
\end{proof}
%\noteWHY{Is the last sentence above a repetition?}

Since $s=W_{(\M_s)}$ for any non-empty state $s$, as a consequence of Lemma \ref{lem.modelstate} and Proposition \ref{prop.irre}, we have:
\begin{proposition}\label{prop.satemodel}
For each non-empty state $s$ in $\M$, any $\alpha\in\LPL$, $\M, s\Vdash\alpha\iff \M_s\vDash \Kh\alpha.$
\end{proposition}
Now we are ready to establish the relation between $\InqL$ and $\InqKhL$.

\begin{theorem} 
\label{ent}
 Given any $\{\alpha\}\cup \Gamma\subseteq \LPL$, $\Gamma\Vdash \alpha$ iff $\Kh\Gamma\vDash \Kh \alpha$.  
\end{theorem} 
\begin{proof}
First, note that due to Proposition \ref{prop.KHK}, $\Kh\Gamma\vDash \Kh \alpha $ iff for any $\M$, $\M\vDash \Kh\Gamma$ implies $\M\vDash \Kh\alpha$.

$\Longrightarrow$: For any model $\M$ s.t.\ $\M\vDash\Kh\Gamma$, we have $W_\M\Vdash\Gamma$ by Lemma \ref{lem.modelstate}. Since $\Gamma\Vdash \alpha$, it follows that $W_\M\Vdash\alpha$. By Lemma \ref{lem.modelstate} again, we have $\M\vDash\Kh\alpha$. Therefore $\Kh\Gamma\vDash \Kh\phi$.

$\Longleftarrow$:  For any model $\M$ and state $s\subseteq W_{\M}$ s.t.\ $s\Vdash\Gamma$, we need to show $s\Vdash \alpha.$ If $s=\varnothing$, since the empty state supports all formulas in inquisitive semantics, it follows that $s\Vdash\alpha$. If $s$ is non-empty, by Proposition \ref{prop.satemodel}, $\M_s\vDash\Kh\Gamma$, thus $\M_s\vDash\Kh\alpha$. By Proposition \ref{prop.satemodel} again, $s\Vdash\alpha$. As a result, $\Gamma\Vdash \alpha$.\end{proof}

When $\Gamma=\emptyset$, it follows immediately that: 
\begin{corollary} \label{coro.equiv}
$\InqL=\InqKhL$.
\end{corollary}

\section{Axiomatizing the full logic} \label{sec.ax}
We showed in the previous section that $\InqKhL=\InqL$, thus the valid $\Kh$-fragment can be axiomatized by the corresponding axioms for inquisitive logic. However, the more interesting question to answer is what the logic with respect to the full language $\DELKh$ is. In this section, we provide a complete axiomatization and also show that the full $\LInqKh$-language is equally expressive as the epistemic fragment with $\K$ modality only.  The conceptual advantage of our axiomatization is that all the axioms are epistemically intuitive, compared to the axioms of inquisitive logic. The axiomatization also shows the hidden dynamic-epistemic content of inquisitive logic in a clear syntactic manner.

%Finally we present the axiomatization of $\textbf{InqKh}$. Note that all $\SFive$ formulas and the syntactic version of the ten reduction rules already give a complete axiomatization of $\textbf{InqKh}$. The following is a more concise one.

%\subsection{The proof system}
\begin{center}
	\begin{tabular}{ll}
		\multicolumn{2}{c}{System $\SDELKh$}\\
		{\textbf{Axioms}}&\\
		\begin{tabular}[t]{lll}
			\TAUT & \text{Propositional tautologies}\\
			\DISTK &$\K (\phi\to\psi)\to (\K\phi\to \K\psi)$\\
			\AxTK&$\K \phi\to \phi$ \\
			\AxTransK&$\K \phi\to\K\K \phi$\\
			\AxEucK&$\neg \K \phi\to\K\neg\K \phi$\\
		    \DISTB &$\Box (\phi\to\psi)\to (\Box\phi\to \Box\psi)$\\
			\AxTB&$\Box \phi\to \phi$ \\
			\AxTransB&$\Box \phi\to\Box\Box \phi$\\
			$\PR$ &$\K\Box\varphi\to\Box\K\varphi$\\
			$\Per$ & $\alpha\to\Box\alpha$ \\
			$\Ver$  & $ \alpha\to\Diamond\Kh\alpha$  \\
			\end{tabular}&
		\begin{tabular}[t]{lll}
            $\KhK$ & $\Kh\alpha\to\K\alpha$\\
	        $\KKhp$& $\K p\to\Kh p$\\
	         $\Khbot$ & $\Kh\bot\lra \bot$\\
            $\KhD$ & $\Kh(\alpha\lor\beta)\leftrightarrow \Kh\alpha\lor\Kh\beta$\\
            $\KhC$ & $\Kh(\alpha\land\beta)\leftrightarrow \Kh\alpha\land\Kh\beta$\\
            $\KhI$ & $\Kh(\alpha\to\beta)\leftrightarrow \K\Box(\Kh\alpha\to\Kh\beta)$\\
            $\AxTransKh$ & $\Kh\alpha\to \K\Kh\alpha$\\
	        $\AxEucKh$& $\neg\Kh\alpha\to \K\neg\Kh\alpha$\\
            $\EU_k$ &$\alpha\land\!\bigwedge_{1\leq i\leq k}\hK(\alpha\land\alpha_i)\!\to\!\Diamond(\K\alpha\land\!\bigwedge_{1\leq i\leq k}\hK\alpha_i)$\\
            & ($k\in\mathbb{N} $,$\alpha_i\in\PL$ for $i\in\mathbb{N}$)\\
            where 
            & $p\in\Prop$, $\alpha,\beta\in \PL$,  $\varphi\in\LInqKh$
		\end{tabular}
	\end{tabular}
	
	\begin{tabular}{lclclc}
		\textbf{Rules:}\\
	\MP & $\dfrac{\varphi,\varphi\to\psi}{\psi}$~~&\NECK&~~ $\dfrac{\vdash\varphi}{\vdash\K\varphi}$ &~~ \NECB&~~$\dfrac{\vdash\varphi}{\vdash \Box \phi}$\\
		\end{tabular}
	\end{center}
	$\SFive$ axiom schemata\slash rules for $\K$ and $\SFour$ axiom schemata\slash rules for $\Box$ are expectable. $\PR$ is the axioms of \textit{perfect recall} often assumed in temporal epistemic logic and dynamic epistemic logic (cf. e.g., \cite{WC13}). $\Per$ says the truth values of propositional formulas do not change given informational updates. $\Ver$ says that propositional truth is eventually verifiable. 
%	\noteWHY{Please check the following descriptions of $\EU_k$.}
	For any finite set of worlds in the current model, $\{\EU_k\mid k\in\mathbb{N}\}$ ensures the existence of an updated submodel that  contains exactly the current world and that set of worlds. %In Proposition \ref{prop. ndkhk} we also showed that for $i\in\mathbb{N}$, $\EU_k$ is equivalent to $\NDkh_k$. 
	We will use $\EU_k$ to prove the reduction formula $\BKD$ in Appendix \ref{sec.app}.
	%$\EU_k$描述了无穷多条公理模式，这组公理模式的直观含义是假如当前模型中能看到一些点，就存在一个子模型，恰好既包含这些点，也包含当前点。值得注意的是，对于任意自然数$k$，$\EU_k$在语义上和$\NDkh_k$等价。这组公式可以帮助我们在公理系统中证明$\BKD$。
	$\KhK$ says that know-how is stronger than know-that. $\KKhp$, $\Khbot$, $\KhD$, $\KhC$, $\KhI$ are the reduction axioms decoding the inquisitive formulas. Introspection schemata $\AxTransK$, $\AxTransKh$ and $\AxEucKh$ can be proved from the rest of the system. In particular, $\AxTransKh$ requires an inductive proof on the structure of $\alpha$. We include them for the sake of their intuitive meanings. 
	
%\noteWHY{I added the case of $\EU_k$.}
\begin{theorem}[Soundness]\label{thm.soundness}
$\SDELKh$ is sound over the class of all epistemic models. 
\end{theorem}
\begin{proof}
The validity of $\SFive$ axiom schemeta\slash rules for $\K$ and $\SFour$ axiom schemata\slash rules for $\Box$ are immediate. Based on Propositions \ref{prop.KHK}, \ref{prop.Kh2K}, \ref{prop.K2Kh}, \ref{prop.vervalidy}, and \ref{prop.khreduction}, we only need to check $\PR$ and $\Per$. For $\PR$,  it is easier to verify its dual form $\Diamond \hK\phi\to \hK\Diamond \phi:$ if there is a submodel where $\phi$ holds at some world  then there is a world and a submodel including it where $\phi$ holds. $\Per$ is valid since moving to any submodel does not change the valuation of the current world.   $\EU_k$ is valid because as long as the current world sees a (finite) set of worlds, we can take the submodel that contains both the current world and the set of worlds as the witness for $\Diamond$.  
\end{proof}

Note that we do not have the rule of uniform substitution for this system in general (recall Example \ref{ex.unisub}). Moreover, even the rule of monotonicity for $\Kh$ is not valid, e.g., $\vDash \neg\neg\alpha\to \alpha$ but $\not\vDash \Kh\neg\neg\alpha\to\Kh\alpha$. However, since we have $\TAUT$ and the modalities of $\Box$ and $\K$ are \textit{normal}, we can have the admissible rule $\rRE$ of  \textit{replacement of equals by equals}  if we treat all the $\Kh\alpha$ as atomic formulas when doing the substitutions: $$\rRE: \qquad  \dfrac{\vdash\varphi\lra\psi}{\vdash \chi[\phi\slash\psi]\lra \chi} \qquad \text{given that the substitution does not happen in the scope of $\Kh$.}$$

Due to $\NECK$, $\NECB$ and the axiom $\KhI$, another useful admissible rule is the syntactic analog of (the non-trivial side of) Proposition \ref{prop.validkh}: 
$$\RKhI: \qquad  \dfrac{\vdash\Kh \alpha\to\Kh\beta}{\vdash \Kh(\alpha\to\beta)} $$

To prove the completeness we need some extra provable (technical) formulas inspired by the reduction of the arbitrary announcement operator in \cite{balbiani2008} in the single-agent case.
\begin{proposition}\label{prop.reduaxbox}The following schemata are provable in $\SDELKh$, where $\alpha\in\LPL$ and $\phi\in \LDEL$ (i.e.,\ $\Kh$-free). 
\begin{center} \label{prop.box}
	\begin{tabular}{ll}
%		\multicolumn{2}{c}{System $\SDELKh$}\\
		\begin{tabular}[t]{lll}
		 \INV &$\Box\alpha\leftrightarrow \alpha$\\
		 \KINV & $\Box\K\alpha\leftrightarrow \K\alpha$\\
		 \hKINV	&$\Box\hK\alpha\leftrightarrow \alpha$\\
			\end{tabular}&
		\begin{tabular}[t]{lll}
				\BD	&$\Box(\alpha\lor\varphi)\leftrightarrow \alpha\lor\Box\varphi$ \\
		 \BKD & $\Box(\hK\alpha\lor\K\alpha_1\lor\dots\lor\K\alpha_n)\leftrightarrow \alpha\lor\K(\alpha\lor\alpha_1)\lor\dots\lor\K(\alpha\lor\alpha_n)$\\
	%	\CRMK & $\Box\Diamond\alpha\leftrightarrow \Diamond\Box\alpha$
		\end{tabular}
	\end{tabular}
% 	\begin{tablular}{lclclc}
% 		\textbf{Rules:}\\
% 	\MP & $\dfrac{\varphi,\varphi\to\psi}{\psi}$~~&\NECK&~~ $\dfrac{\vdash\varphi}{\vdash\K\varphi}$ &~~ \NECB&~~$\dfrac{\vdash\varphi}{\vdash \Box \phi}$\\
% 		\end{tablular}
	\end{center}
\end{proposition}
\begin{proof}
    \INV\ is the combination of  $\Per$ and $\AxTB$. \KINV\ is proved from $\Per$ and $\PR$. \hKINV\ is a special instance of $\BKD$ when $n=0$. We include the (tedious) proofs of the other two formulas in Appendix \ref{sec.app}.    
\end{proof}

Recall that $\LDEL$ is the $\Kh$-free fragment of $\LInqKh$, and $\LEL$ is the $\Box$-free fragment of $\LDEL$. By the following lemmata, we show that each $\LInqKh$-formula is provably equivalent to an $\LEL$-formula. 
\begin{lemma} \label{lem.redkh}
Each $\LInqKh$-formula is provably equivalent to a $\LDEL$ formula in $\SDELKh$. 
\end{lemma}
\begin{proof}
Note that with the help of $\rRE$, we can repeatedly apply Axioms $\Khbot, \KhC, \KhD, \KhI$ step by step to reduce $\Kh\alpha$ to simpler $\Kh$ formulas. It is not hard to show that eventually, all the $\Kh$-formula can be reduced to some formulas with $\Kh p$ only. By Axioms $\KhK$ and $\KKhp$, we have $\vdash \Kh p\lra \K p$, which will eventually eliminate any $\Kh$ modality completely. 
\end{proof}

\begin{lemma}\label{lem.redbox}
Each $\LDEL$-formula is provably equivalent to an $\LEL$ formula in $\SDELKh$. 
\end{lemma}
\begin{proof}
Here we follow the idea in \cite{balbiani2008}. We say a formula $\varphi\in\LEL$ is in normal form if it is a conjunction of disjunctions of the form $\alpha\lor\hK\alpha_0\lor\K\alpha_1\lor\cdots\lor\K\alpha_n$ where $\alpha, \alpha_0,\cdots,\alpha_n\in\LPL$. Every formula in single-agent $\SFive$ is equivalent to a formula in normal form \cite{Meyer1995}. So we only need to prove that $\Box(\alpha\lor\hK\alpha_0\lor\K\alpha_1\lor\cdots\lor\K\alpha_n)$ is provably equivalent to an $\LEL$ formula in $\SDELKh$. Then we can eliminate the $\Box$ step by step from the innermost ones that do not have any $\Box$ in its scope. 

By $\BD$, $\Box(\alpha\lor\hK\alpha_0\lor\K\alpha_1\lor\cdots\lor\K\alpha_n)$ is equivalent to $\alpha\lor\Box(\hK\alpha_0\lor\K\alpha_1\lor\cdots\lor\K\alpha_n)$, and by $\BKD$, $\alpha\lor\Box(\hK\alpha_0\lor\K\alpha_1\lor\cdots\lor\K\alpha_n)$ is equivalent to $\alpha\lor\alpha_0\lor\K(\alpha_0\lor\alpha_1)\lor\cdots\lor\K(\alpha_0\lor\alpha_n)$, which is a formula in $\LEL$. Note that $\BD$ and $\BKD$ are provable in $\SDELKh$ as shown by Proposition \ref{prop.reduaxbox}. 

%    \noteYW{Need a proof sketch using ideas in \cite{balbiani2008}.}
\end{proof}

Lemmata \ref{lem.redkh} and \ref{lem.redbox} together with Theorem \ref{thm.soundness} also tell us about the expressivity of $\LInqKh$.
\begin{theorem}[Expressivity]\label{thm.exp}
$\LInqKh$ is equally expressive as $\LEL$ over epistemic models. 
\end{theorem}
% If we view states as epistemic models, then we can compare the expressivity of $\InqKhL$ and
% \begin{corollary}

% \end{corollary}
In particular, we have the following corollary. 

\begin{corollary}
For each $\alpha\in \LPL$, there is an epistemic formula $\phi\in \LEL$ such that $$\M,W_\M \Vdash \alpha \iff \M,w\vDash \Kh\alpha \iff \M,w\vDash \phi.$$ 
\end{corollary}

% Actually, as shown in the literature, we can define a translation $t$ to translate each $\LInqKh$ formula to an equivalent $\LEL$-formula. 

As shown in the literature, a simple direct translation from inquisitive logic to epistemic logic will be given in Section  \ref{sec.related work}. 

\begin{theorem}[Completeness]
$\SDELKh$ is complete over the class of all epistemic models. 
\end{theorem}
\begin{proof}
  We prove the completeness by translating each  $\LInqKh$-formula $\phi$ into an equivalent $\LEL$-formula $\phi'$ and follow the strategy below.  
  $$\vDash\phi\stackrel{\text{expressive equivalence}}{\implies}\vDash \phi'\stackrel{\text{completeness of }\SFive }{\implies}\vdash_{\SFive }\phi'\stackrel{\text{\SFive $\subseteq \SDELKh$}}{\implies}\vdash_{\SDELKh} \phi'\stackrel{\text{provable equivalence}}{\implies}\vdash_{\SDELKh}\phi$$  
\end{proof}
Strong completeness can be obtained in the similar  process w.r.t. a given assumption set $\Gamma$ (cf. e.g., \cite{baltag2016logic}). Decidability of $\SDELKh$ immediately follows from the proofs of the above theorems and the decidability of $\SFive$ . 
\begin{corollary}
$\SDELKh$ is decidable. 
\end{corollary}

To demonstrate the power and the use of our system, we show that the following important axioms can be proved based on our intuitive epistemic axioms. In particular, compared to the semantic validity of $\DNkh$ shown in Proposition \ref{prop.axvalidity}, the syntactic proof of $\DNkh$ presents the non-trivial use of the axioms regarding $\Box$ and $\K$ in Proposition \ref{prop.reduaxbox}.   
\begin{proposition}
The following are provable in $\SDELKh$: 
\begin{center}
\begin{tabular}{ll}
$\KKhN$ &$\K\neg\alpha\to \Kh \neg \alpha$\\
$\DNkh$ &$\Kh(\neg\neg p\to p)$\\
%$\CRMK$ & $\Box\Diamond\alpha\leftrightarrow \Diamond\Box\alpha$ \\
% $\Peirce$ &$\Kh(((p\to q)\to p)\to p)$\\
% $\KP$ & $\Kh (\neg\alpha\to\beta\lor\gamma)\to(\neg\alpha\to\beta)\lor(\neg\alpha\to\gamma)$\\
$\NDkh_k$ & $\Kh((\neg \alpha \to \bigvee_{1\leq i\leq k}
\neg \beta_i)\to \bigvee_{1\leq i\leq k}(\neg \alpha \to \neg \beta_i))$
%\noteWHY{I changed $p_i$ for $\alpha_i$.}
\end{tabular}
\end{center}
\end{proposition}
\begin{proof}
%\noteYW{To provide the proofs.}
\text{For $\KKhN$:}
{\setcounter{equation}{0}}
\begin{align}
	&\vdash\Kh\alpha\to \alpha\ &&\KhK,\AxTK,\TAUT\label{106}\\
	&\vdash\neg\alpha\to \neg\Kh\alpha\ &&(\ref{106})\TAUT\label{107}\\
	&\vdash\Box\neg\alpha\to\Box\neg\Kh\alpha &&(\ref{107})\NECB\label{108}\\
	&\vdash\neg\alpha\to\Box\neg\Kh\alpha &&(\ref{108})\Per,\TAUT\label{109}\\
	&\vdash\K\neg\alpha\to\K\Box\neg\Kh\alpha &&(\ref{109})\NECK\label{110}\\
	&\vdash\K\neg\alpha\to\K\Box(\Kh\alpha\to\Kh\bot) &&(\ref{110})\Khbot,\rRE\label{111}\\
	&\vdash\K\neg\alpha\to \Kh \neg \alpha &&(\ref{111})\KhI,\rRE
\end{align}
\text{For $\DNkh$:}
\begin{align}
{\setcounter{equation}{0}}
	&\vdash\neg\neg p\to p &&\TAUT\label{100}\\
	&\vdash\K(\neg\neg p\to  p) &&(\ref{100})\NECK\label{101}\\
	&\vdash\K\neg\neg p\to \K p &&(\ref{101})\DISTK, \MP\label{102}\\
    &\vdash\Kh\neg\neg p\to \K\neg\neg p &&\KhK\label{103}\\
    &\vdash\K p\to \Kh p 
    &&\KKhp\label{104}\\
	&\vdash\Kh\neg\neg p\to \Kh p &&(\ref{103})(\ref{102})(\ref{104})\TAUT\label{105}\\
	&\vdash\Kh(\neg\neg p\to p) &&(\ref{105})\RKhI
\end{align}
\text{For $\NDkh_k$:}
%\noteWHY{I have some problems with the proof. However we can prove $\K\neg\alpha\to \Kh \neg \alpha$.}
\begin{align}
{\setcounter{equation}{0}}
    &\vdash\K\alpha\lor\bigvee_{1\leq i\leq k}
    \K(\alpha\lor\neg \beta_i)\to \bigvee_{1\leq i\leq k}(\K\alpha\lor\K(\alpha\lor\neg\beta_i))            
    &&\TAUT\label{121}	\\
    &\vdash\K(\alpha\lor\bigvee_{1\leq i\leq k}\K(\alpha\lor\neg \beta_i))\lra (\K\alpha\lor\bigvee_{1\leq i\leq k}\K(\alpha\lor\neg \beta_i))   
    &&\SFive_{\K}\label{122}	\\
% 	&\vdash\K\alpha\lor\bigvee_{1\leq i\leq k}
%     \K\neg \beta_i\to \K\alpha\lor\bigvee_{1\leq i\leq k}\K(\alpha\lor\neg\beta_i)     
%     &&\TAUT\label{121}	\\
%     &\vdash\K(\alpha\lor\bigvee_{1\leq i\leq k}\K\neg \beta_i)\lra (\K\alpha\lor\bigvee_{1\leq i\leq k}\K\neg \beta_i)    
%     &&\SFive_{\K}\label{122}	\\
    &\vdash\K(\alpha \lor\K(\alpha\lor\neg\beta_i)) \lra (\K\alpha\lor\K(\alpha\lor\neg\beta_i))    &&\SFive_{\K}\label{123}	\\
    &\vdash\K(\alpha\lor\bigvee_{1\leq i\leq k}\K(\alpha \lor\neg \beta_i))\to \bigvee_{1\leq i\leq k}\K(\alpha \lor\K(\alpha\lor\neg\beta_i))  &&(\ref{121})(\ref{122})(\ref{123})\rRE\label{124}	\\
    &\vdash\Box(\hK \alpha \lor \bigvee_{1\leq i\leq k}
    \K\neg \beta_i)\lra (\alpha\lor\bigvee_{1\leq i\leq k}\K(\alpha \lor\neg \beta_i))  
    &&\BKD\label{125}	\\   
    &\vdash\Box(\hK \alpha \lor \K\neg \beta_i)\lra (\alpha \lor\K(\alpha\lor\neg\beta_i))
    &&\BKD\label{126}	\\
    %&\vdash\Box(\K\neg \alpha \to \bigvee_{1\leq i\leq k}
    %\K\neg \beta_i)\to \bigvee_{1\leq i\leq k}\Box(\K\neg \alpha \to \K\neg \beta_i)    %&&(\ref{124})(\ref{125})(\ref{126})\rRE\label{127}	\\
    &\vdash\K\Box(\K\neg \alpha \to \bigvee_{1\leq i\leq k}
    \K\neg \beta_i)\to \bigvee_{1\leq i\leq k}\K\Box(\K\neg \alpha \to \K\neg \beta_i)    &&(\ref{124})(\ref{125})(\ref{126})\rRE\label{128}	\\
    &\vdash\K\Box(\Kh\neg \alpha \to \bigvee_{1\leq i\leq k}
    \Kh\neg \beta_i)\to \bigvee_{1\leq i\leq k}\K\Box(\Kh\neg \alpha \to \Kh\neg \beta_i)   &&(\ref{128})\KKhN,\KhK,\rRE\label{129}	\\
    &\vdash\K\Box(\Kh\neg \alpha \to \Kh\bigvee_{1\leq i\leq k}
    \neg \beta_i)\to \bigvee_{1\leq i\leq k}\Kh(\neg \alpha \to \neg \beta_i)     
    &&(\ref{129})\KhD,\KhI, \rRE\label{130}\\
    &\vdash\Kh(\neg \alpha \to \bigvee_{1\leq i\leq k}
    \neg \beta_i)\to \Kh\bigvee_{1\leq i\leq k}(\neg \alpha \to \neg \beta_i)     
    &&(\ref{130})\KhI,\KhD, \rRE\label{131}\\
    &\vdash\Kh((\neg \alpha \to \bigvee_{1\leq i\leq k}
    \neg \beta_i)\to \bigvee_{1\leq i\leq k}(\neg \alpha \to \neg \beta_i)) 
    &&(\ref{131})\RKhI,\rRE
\end{align}

% \noteYW{ What about $\Box\Diamond\phi\leftrightarrow \Diamond\Box\phi$? If it is not easy to show then we can probably omit it.} \\
% $\CRMK$ is the combination of Church-Rosser and McKinsey axioms which can be proved by using \INV\ and the replacement of equals:
%     \begin{align}
%     &\Box\alpha\leftrightarrow \alpha&& \INV\label{57}\\
%         &\Box\neg\alpha\leftrightarrow \neg\alpha&& \INV\label{577}\\
% 	&\neg\Diamond\neg\alpha\leftrightarrow \neg\neg\alpha&&(\ref{57}),\TAUT, \rRE\label{58}\\
% 	&\neg\Diamond\Box\neg\alpha\leftrightarrow \neg\Box\neg\alpha&&(\ref{58}), (\ref{577}),\rRE\label{59}\\
% 	&\Box\Diamond\alpha\leftrightarrow \Diamond\alpha&&(\ref{59}), \TAUT, \rRE\label{60}\\
% 	&\Box\Diamond\alpha\leftrightarrow \Diamond\Box\alpha&&(\ref{60}), (\ref{57}), \rRE
% \end{align}
\end{proof}
% \begin{proposition} For $\alpha\in\LPL$, we have $\vdash\Box\Diamond\alpha\leftrightarrow \Diamond\Box\alpha$
% \end{proposition}
%%% the proof for \CRM:     
% \begin{align}
%     &\Box\alpha\leftrightarrow \alpha&&(\text{Proposition }\ref{prop.box})\label{57}\\
% 	&\neg\Diamond\neg\alpha\leftrightarrow \neg\neg\alpha&&(\ref{57})PL\label{58}\\
% 	&\neg\Diamond\Box\neg\alpha\leftrightarrow \neg\Box\neg\alpha&&(\ref{58})(\text{Proposition }\ref{prop.box})PL\label{59}\\
% 	&\Box\Diamond\alpha\leftrightarrow \Diamond\alpha&&(\ref{59})PL\label{60}\\
% 	&\Box\Diamond\alpha\leftrightarrow \Diamond\Box\alpha&&(\ref{60})(\text{Proposition }\ref{prop.box})PL
% \end{align}
Note that although $\NDkh$ and $\KhKP$ are very similar semantically as in the proof of Proposition \ref{prop.axvalidity}, deriving $\KhKP$ requires more efforts within our proof system due to the absence of negations in front of $\beta_i$, which bridged $\K$ and $\Kh$ in the above proof. Nevertheless, $\KhKP$ can be proved via a detour using Theorem \ref{thm.transs5} that gives the equivalent $\bigvee \K\rho_i$ form of each $\Kh$-formula, which then reduces the case to $\NDkh$.\footnote{Instead of $\EU_k$, we may try to add an axiom $\hK\Diamond(\K\alpha \land  \neg \Kh\beta)\land \hK\Diamond(\K\alpha \land \neg \Kh\gamma)\to \hK\Diamond (\K\alpha\land \neg \Kh\beta\land \neg \Kh\gamma)$, which captures the possibility of merging two submodels, thus having a direct link with $\KhKP$ (cf. the proof of Proposition \ref{prop.axvalidity}.)  
}

To end this section, we discuss the connections of our logic of knowing how with the \textit{planning-based} knowing how logic studied in \cite{Wang2016,Fervari2017,Li2021-LIPKH}. In \cite{Li2021-LIPKH}, $\Kh\phi$ roughly says that there is a plan such that I know that executing it will always guarantee $\phi$, where $\phi$ can be any formula in the language. Note that this is very different from the intuitive reading of know-how operator in the current framework, where $\Kh\alpha$ says that knowing how $\alpha$ is true, where $\alpha$ is a propositional formula. The fundamental difference in the semantics is reflected by the axioms regarding the $\Kh$ modality. For example, $\Kh(\phi\lor\psi)\to (\Kh\phi\lor\Kh\psi)$ is not valid in  \cite{Fervari2017,Li2021-LIPKH}, e.g., one can always use any  plan to make sure $p\lor\neg p$ but it does not mean one can make sure $p$ or make sure $\neg p$. On the other hand, $\K\phi\to \Kh\phi$ is valid in \cite{Fervari2017,Li2021-LIPKH} since you can always use the empty plan, whereas in our framework, it does not hold in general when $\phi$ is not a statement.

\section{Epistemic interpretation of inquisitive semantics}\label{sec.understanding}

%\noteYW{I rewrote the section using the new notations and definitions. Please double check.}

%\subsection{Formalizing concepts in inquisitive semantics}
In this section, we look at the concepts in inquisitive semantics from our alternative epistemic perspective. 
%We will only consider distinguishing epistemic models, which can be viewed as the states in the original inquisitive semantics. Note that as shown by Proposition \ref{prop.disM}, this is not an essential restriction.  

We first summarize the correspondence between our semantics and inquisitive semantics below based on the corresponding definitions in Section \ref{sec.pre}.\footnote{If $V[s]=\wp(\Prop)$ then we have the corresponding \textit{absolute} notions of inquisitiveness and informativeness as in \cite{Ciardelli2011}.}
\begin{center}
    \begin{tabular}{|l|l|}
    \hline
  Inquisitive semantics & Our epistemic semantics  \\
  \hline
  Information model & single-agent S5 epistemic model (with an implicit total relation)\\ 
  non-empty states & epistemic submodels\\ 
  support ($\M, s\Vdash\alpha$) &  know-how ($\M_s\vDash\Kh\alpha$) \\
  alternatives for $\alpha$ in $\M$ & maximal submodels of $\M$ satisfying $\Kh\alpha$\\
  proposition expressed by $\alpha$ in $\M$ & set of submodels of $\M$ for $\Kh\alpha$  \\
  $\alpha$ is inquisitive in $\M$ & there are two maximal submodels satisfying $\Kh\alpha$ \\
  $\alpha$ is informative in $\M$ & there is one world not in any maximal submodels of $\M$ for $\Kh\alpha$ \\
  \hline 
\end{tabular}
\end{center}
% Furthermore, we can define the relative notions of \textit{statements} and \textit{questions}.\footnote{As mentioned, the term ``assertion'' has been replaced by ``statement'' in the latest literature. We adopt the latest terminology here. Note that in Definition 2.14 in \cite{Ciardelli2011}, questions and statements (assertions) are defined absolutely with respect to a full model with the trivial state. Here we generalized it to a relative notion as in the cases of inquisitiveness and informativeness. } 
% \begin{definition}[Questions and statements]\,
% Given a model $\M$: 
% \begin{itemize}
%  \item $\alpha$ is a \emph{question} in $\M$ iff it is not informative in $\M$;
% \item $\alpha$ is a \emph{statement} in $\M$ iff it is not inquisitive in $\M$.
% \end{itemize}
% \end{definition}
% According to the definitions, we have the derived notions: 
% \begin{center}
%     \begin{tabular}{l@{$\iff$ }l}
%     $\alpha$ is a \emph{question} in $s$ & it is \textit{not} informative in $s$\\
%     $\alpha$ is an \emph{statement} in $s$ & it is \textit{not} inquisitive in $s$
%     \end{tabular}
% \end{center}

\medskip

First we give an epistemic characterization of informativeness (and questions). The idea is that $\alpha$ is informative iff you do not know that $\alpha$ already.
\begin{proposition}\label{prop.ELinfo}Given any $\M$,  $\alpha$ is informative in $\M$ iff $\M\vDash \neg\K\alpha$. Thus $\alpha$ is a question in $\M$ iff $\M\vDash \K\alpha$.
\end{proposition}
\begin{proof}
Recall that $\alpha$ is informative in $\M$ iff there is at least one world in $\M$ that is not included in any alternative for $\alpha$ in $\M$ (cf. Definition \ref{def.inqandinfo}). This definition can be rendered intuitively (and formally) as $\M\vDash \hK \neg \Diamond \Kh\alpha$ in our framework. Now due to Proposition \ref{prop.vervalidy}, it is equivalent to $\M\vDash \hK\neg \alpha,$ namely, $\M\vDash\neg\K \alpha.$ 
%\noteYW{say something?}
%\noteWHY{reference to the proposition about singleton model needs to be added; I have not write anything about $\M_s\vDash \neg\K\alpha$ in the previous versions.}
%Recall that $\alpha$ is informative in $s$ iff there is at least one world in $s$, let it be $w$, that is not included in
%any possibility for $\alpha$ in $s$. where by Definition \ref{def.possibility}, a \emph{possibility} for $\alpha$ in $s$ is a maximal substate of $s$ supporting $\alpha$. 
% $\Longrightarrow$ Suppose $\alpha$ is informative, it follows by Proposition \ref{prop.sup.pos} that $\{w\}\not\Vdash\alpha$.
% Since $\M_s=\langle s,\V\rangle$, we have $w\in \M_s$. Therefore $\M_{\{w\}}\subseteq M_s$ and $\M_{\{w\}},w\not\vDash\Kh\alpha$. Since $\M_{\{w\}}$ is a singleton model, this is equivalent to $\M_s,w\not\vDash\alpha$, which means $\M_s\vDash \neg\K\alpha$. 
% $\Longleftarrow$ Suppose $\M_s\vDash \neg\K\alpha$. Then there is a $w\in M_s$ such that $\{w\},w\not\vDash\alpha$, which is followed by $\{w\},w\not\vDash\Kh\alpha$. It follows that $s_\{w\}\not\Vdash\alpha$. By Proposiiton \ref{prop.sup.pos},  $\alpha$ is informative.
% The other half of the proposition follows immediately by definition.
\end{proof}
Next, we prove an epistemic characterization of inquisitiveness and statements. The idea is that $\alpha$ is inquisitive in $s$ iff it is possible to know $\alpha$ while not knowing how to resolve it with the information provided by $s$ in a model $\M$.
%\noteWHY{I changed the phrasing.}
\begin{proposition}\label{prop.ELInq} Given any $\M$, 
$\alpha$ is inquisitive in $\M$ iff $\M \vDash\hK\Diamond(\K\alpha\land\neg\Kh\alpha).$ Thus $\alpha$ is a statement in $\M$ iff $\M\vDash \K\Box(\K\alpha\to\Kh\alpha). $ 
%In particular, if $\alpha$ is a statement in $s$ then $\M_s\vDash\K\alpha\lra\Kh\alpha.$
\end{proposition}
\begin{proof} 
Recall that $\alpha$ is inquisitive in $\M$ iff there are at least \textit{two} alternatives for $\alpha$ (cf.\ Definition \ref{def.inqandinfo}). In our terms, it means there is some submodel (e.g., the union of those two alternatives) such that $\K\alpha$ holds but there is no uniform resolution. Formally, it amounts to $\M\vDash \hK\Diamond(\K\alpha\land\neg\Kh\alpha)$. Therefore $\alpha$ is a statement in $\M$ iff $\alpha$ is not inquisitive iff $\M\vDash \neg \hK\Diamond(\K\alpha\land\neg\Kh\alpha)$ iff $\M\vDash \K\Box(\K\alpha\to\Kh\alpha)$.
 
% For any $\alpha\in\LPL$ and pointed model $\M,w$ of $\LInqKh$. 

% $\Longrightarrow$: Let $\M'=\M_1\cup \M_2$, there are $\M_1, \M_2\subseteq \M$ s.t. $\M_1\vDash{\Kh}\alpha$, $\M_2\vDash{\Kh}\alpha$, and $\M'\not\vDash{\Kh}\alpha$. By 1.3.3, $\M_1\vDash{\K}\alpha$, $\M_2\vDash{\K}\alpha$, and thus $\M'\vDash{\K}\alpha$. Since $\M_1, \M_2$ are non-empty, $\M'$ is non-empty. So there exists a $v\in \M'\subseteq \M$, $v\sim w$ and $\M,v\vDash \Diamond(\K\alpha\land\neg\Kh\alpha)$, i.e. $\M, w\vDash\hK\Diamond(\K\alpha\land\neg\Kh\alpha)$.

% $\Longleftarrow$: Since $\M, w\vDash\hK\Diamond(\K\alpha\land\neg\Kh\alpha)$, there exists a $v\in \M'\subseteq \M$, $v\sim w$ and $\M',v\vDash \K\alpha\land\neg\Kh\alpha$. Since $\M',v\vDash \K\alpha$, for any $v\in M'$, $\R(v,\alpha)\neq\varnothing$. For any $x\in B=\bigcup_{v\in M'}\{ \R(v,\alpha)\}$.%R有穷所以不需要AC
% Let $A_x=\{w\in M'\mid x\in  \R(w,\alpha)\}$. $A_x\neq M'$, otherwise $x$ guarantees that $\M',v\vDash \K\alpha\land\neg\Kh\alpha$. Since $\{A_x\mid x\in B\}\in\wp(\M')$, $(\{A_x\mid x\in B\},\subseteq)$ forms a partial order. So we can choose a maximal element $A_z\in\{A_x\mid x\in B\}$ and an $s\in \M'\backslash A_z$. Let $\M_1=A_z$, $\M_2=\{s\}$. By the existence of $z$ $\M_1\vDash{\Kh}\alpha$. Since $\{s\},s\vDash\alpha$,  $\M_2\vDash{\Kh}\alpha$. And by the maximality of $A_z$, $\M_1\cup \M_2\not\vDash{\Kh}\alpha$.
\end{proof}
Note that $\K\Box (\K\alpha\to\Kh\alpha)$ is equivalent to $\Kh(\neg\neg\alpha\to \alpha)$ which is consistent with the characterization of statements in \cite{Ciardelli2011}. $\K\Box (\K\alpha\to\Kh\alpha)$ says that I know that if you tell me $\alpha$ is resolvable then I will know how to resolve it. 

In particular, when taking a full model $\M$ and the trivial state $W_\M$, then from Propositions \ref{prop.full}, \ref{prop.ELinfo} and \ref{prop.ELInq}, we have the characterization of absolute questions and statements. 
%\noteYW{I added the Proposition \ref{prop.full}.}
\begin{proposition} \label{prop.assert}
 $\alpha$ is a question iff $\K\alpha$ is valid iff $\alpha$ is a classical tautology. $\alpha$ is a statement iff $\K\alpha\to \Kh\alpha$ is valid iff  $\K\alpha\lra \Kh\alpha$ is valid.
\end{proposition}

Based on the above epistemic characterization of the relative notion of  statements, we can prove the following from a purely epistemic perspective extending the result in Proposition 2.19 in \cite{Ciardelli2011} about absolute statements. 
\begin{proposition}
For any $p\in\Prop$ and $\alpha,\beta\in\LPL$, any  given model $\M$:
\begin{enumerate}
    \item $p$ is a statement in $\M$.
    \item $\bot$ is a statement in $\M$.
    \item If $\alpha$ and $\beta$ are statements in $\M$, then $\alpha\land\beta$ is a statement in $\M$.
    \item if $\beta$ is a statement in $\M$, then $\alpha\to\beta$ is a statement in $\M$.
\end{enumerate}
\end{proposition} 
\begin{proof}
(1) and (2) are trivial. 
For (3): Suppose $\alpha$ and $\beta$ are statements in $\M$, by Proposition \ref{prop.ELInq}, we need to show $\M\vDash \K\Box (\K(\alpha\land\beta)\to\Kh(\alpha\land\beta))$, which is equivalent to $\M\vDash \K\Box ((\K\alpha\land\K\beta)\to(\Kh\alpha\land\Kh\beta))$. By the assumption that $\alpha$ and $\beta$ are statements in $\M$, it can be easily proved.

Now for (4): Suppose that $\beta$ is a statement in $\M$. By Proposition \ref{prop.ELInq}, we have $\M\vDash\K\Box (\K\beta\to \Kh\beta)\ (\dagger).$ We need to show that $\M\vDash \K\Box (\K(\alpha\to\beta)\to\Kh(\alpha\to\beta))$,  namely, for any submodel $\M'$ of $\M$ if $\M'\vDash\K(\alpha\to\beta)$ then $\M'\vDash \K\Box(\Kh\alpha\to \Kh\beta)$. Suppose  $\M'\vDash \K(\alpha\to\beta)$ and take any submodel $\M''$ of $\M'$ such that $\M''\vDash\Kh\alpha$, we only need to show $\M''\vDash \Kh\beta$. Since $\vDash\Kh\alpha\to\K\alpha$, we have $\M''\vDash \K\alpha$. Since $\M'\vDash \K(\alpha\to\beta)$ and $\M''$ is a submodel of $\M'$, we have $\M''\vDash \K\beta. $ Now due to the fact that $\M''$ is also a submodel of $\M$, we have $\M''\vDash\Kh\beta$ by $(\dagger)$, which completes the proof.   % \begin{enumerate}
% \item $\M, w\vDash\K(\alpha\land\beta)\implies \M, w\vDash\K\alpha\text{ and } \M, w\vDash\K\beta\implies \M, w\vDash\Kh\alpha\text{ and } \M, w\vDash\Kh\beta\implies \M, w\vDash\Kh(\alpha\land\beta)$
%     \item Suppose $\M, w\vDash\K(\alpha\to\beta)$, then for any $v\in M$, $\M, v\vDash\alpha$ implies $\M, v\vDash\beta$. By Proposition \ref{prop.nept}, that is $\R(v,\alpha)\neq\varnothing$ implies $\R(v,\beta)\neq\varnothing$. Let $\M'=\langle W',\sim\,,\R$ be the subset of $\M$ s.t. $W'=\{w\mid \R(w,\alpha)\neq\varnothing\}$. Thus for any $w\in M'$, we have $\M',w\vDash\K\alpha$. Therefore $\M',w\vDash\K\beta$. Since $\beta$ is a statement, $\M',w\vDash\Kh\beta$. Therefore for each $w\in M'$, there exists an $x$ s.t. $x\in\R(\M',\beta)$. Let $f\in\S(\beta)^{\S(\alpha)}$ be a function s.t. $f(y)=x$ for all $y\in\S(\alpha)$. Since $f[\R(w,\alpha)]\subseteq\R(w,\beta)$ for all $x\in M$, $f\in\R(\M,\alpha\to\beta)$. So $\M, w\vDash\Kh(\alpha\to\beta)$
% \end{enumerate}
\end{proof}

As an immediate consequence, any disjunction-free formula is a statement in any $\M$.
\begin{proposition}\label{prop.disjfree}
For any $\alpha\in\LPL$, if $\alpha$ is disjunction-free, then $\alpha$ is a statement in $\M$ for any $\M$. 
\end{proposition} 

Next, we show the epistemic proof of another result in \cite{Ciardelli2011} linking the three most important concepts in inquisitive semantics. 

\begin{proposition}[Support, inquisitiveness, and informativeness] $\M,s$ supports
a formula $\alpha$ if $\alpha$ is neither inquisitive in $\M_s$ nor informative in $\M_s$.
\end{proposition}
\begin{proof}It is straightforward in our case due to the validity of 
$\Kh \alpha \lra (\K\alpha\land \K\Box(\K\alpha\to \Kh\alpha)). $ The right-to-left implication is due to the validity of the axioms $\AxTB$ and $\AxTK$. The left-to-right implication is due to the validity of $\KhK$ and the fact that once we have a uniform resolution for $\alpha$ then we still have it in any submodel. 
%Indeed, if $\Kh\alpha$ is true on a model $\M$, then there is a uniform resolution, say $x$, for $\alpha$ on $\M$, which implies that there are the resolution of $\alpha$ on each point in $\M$. Also, for any submodel of $\M$, $x$ is a uniform resolution for $\alpha$, which means that $\K\Box\Kh\alpha$ holds. So $\K\Box(\K\alpha\to\Kh\alpha)$ trivially holds on $\M$. The other way around, since any model is the submodel of itself, $\K\alpha\land\K\Box(\K\alpha\to\Kh\alpha)$ implies $\K\alpha\land(\K\alpha\to\Kh\alpha)$. Therefore $\Kh\alpha$ holds on that model.
\end{proof}

%\noteYW{To be written. }

As we mentioned earlier, the intended interpretation of inquisitive logic is not about knowledge, and what we are presenting is an alternative interpretation. On the other hand, there are intimate connections between the two interpretations conceptually. We conclude the section by the following discussion on knowledge and information range.
%\noteWHY{I changed ``the thing to be ascribed with knowledge statement is not an agent at all'' into ``the object to which we ascribe the knowledge statement is not an agent at all''.}
Starting from Hintikka \cite{Hintikka:kab}, knowledge in epistemic logic is defined based on information range: $\phi$ is known iff in $\phi$ holds on all the epistemic alternatives of the real world. This notion of knowledge is applied to various fields where the object to which we ascribe the knowledge statement is not an agent at all, such as in distributed system of computer science. Note that in many applications, the so-called ``agent'' is just a way of talking about the information range. It is also reflected in the technical fact that the classical propositional reasoning can be cast in the epistemic logic such that $\alpha\vDash_\CPL\beta$ iff $\K\alpha\vDash_{\mathbf{S5}} \K\beta$. Intuitively we can turn a classical propositional entailment from $\alpha$ to $\beta$ into an epistemic entailment by \textit{if $\alpha$ is known, then $\beta$ is also known}. As we have seen in Theorem~\ref{ent}, we can generalize it to match the entailment of $\InqL$ with then entailment in our  know-how variant of epistemic logic. Moreover, when talking about general knowledge, it is not essential to ask which specific agent the knowledge is ascribed to. Actually, the basic systems of epistemic logic are suitable for this kind of general reasoning, by assuming an idealized  reasoner, no matter \textit{who} exactly knows. From our perspective, when we do not give any special constraints on the information that the agent has, we can actually talk about the general semantic knowledge which is not ascribed to a particular agent. Therefore, to us, there is no drastic conceptual gap between the information model for the inquisitive semantics and epistemic model in our setting. Nevertheless, in concrete applications about real agents, some more detailed constraints on what agents know and what they can learn may be given, which may result in changes of logic (cf. e.g., \cite[Section 2.8]{draftCiardelli}).  

%\noteYW{To add some discussion between epistemics and information and Hintikka's eary work.}

%\noteYW{Maybe some entailment example between questions and declaratives..}

% \begin{definition}[Question]
% $\alpha\in\LPL$ is a question iff it is not informative, i.e. $\alpha$ provides no information. Formally, $\alpha$ is a question iff for any $w\in W_{full}$, $\M_{full}, w\vDash\K\alpha$.
% \end{definition}

% \begin{definition}[Possibility]
% For any $\alpha\in\LPL$, and a model $\M$, a model $\M'\subseteq \M$ is a possibility for $\alpha$ in $\M$ iff $\M'$ is a maximal model s.t. $\M'\vDash\Kh\alpha$, i.e. $\M'\vDash\Kh\alpha$, and for any $\M''\subseteq \M$ s.t. $\M'\subseteq \M''$, $\M''\not\vDash\Kh\alpha$.
% \end{definition}

% \begin{definition}[Proposition]
% For any $\alpha\in\LPL$, we define its corresponding proposition $[\alpha]$ as $[\alpha]=\{M\mid M\text{ is a possibility for }\alpha\}$.
% \end{definition}

% \begin{proposition}[Characterization of Proposition]
% $\alpha\in\LPL$ is informative relative to $\M,w$ iff $\M, w\vDash\neg\K\alpha$.
% 	\begin{itemize}
% 			\item [1.] $[p]=\{\mid p\mid\}$ for $p\in\Prop$
% 			\item [2.] $[\bot]=\varnothing$
% 			\item [3.] $[\alpha\lor\beta]=[\alpha]\cup[\beta]$
% 			\item [4.] $[\alpha\land\beta]=\{A\cap B\mid A\in[\alpha]\text{ and } B\in[\beta]\}$
% 			\item [5.] $[\alpha\to\beta]=\{\bigcap_{A\in[\alpha]}((M_{full}\backslash A)\cup B_A)\mid B_A\in [\beta]\}$
% 		\end{itemize}
% \end{proposition}

\section{Related work}
\label{sec.related work}
In this section, we first connect the formula-based resolutions studied in the literature of inquisitive logic in  \cite{Ciardelli2011,Ciardelli2018} to our approach, and then compare our work to the related modal logic work in the literature. 

\subsection{Resolutions in terms of formulas in $\LPL$}
In our framework, the resolutions defined in Definitions \ref{def.rs} and \ref{def.resolution} are not in the object language. It is not hard to check as long as $S(p)$ is a singleton $\{x\}$ and $R(w, p)=S(p)$ iff $p\in V_\M(w)$, the semantics will not affect the logic.\footnote{Note that what exactly \textit{is} the resolution of each atomic proposition does not matter, due to our BHK-like semantics for $\Kh$, which only checks the existence of (uniform) resolutions.} On the other hand, since the resolutions of atomic propositions are defined simply as themselves, it is possible to express the resolutions using formulas in the object language. In \cite{Ciardelli2011}, the authors proposed a notion of resolutions expressed by disjunction-free formulas in $\PL$. Here we take the simplified version Definition 8 in  \cite{Ciardelli2018}:\footnote{In Definition 9.1 in \cite{Ciardelli2011}, the resolution formulas (realizations) are based on some normal form. }  %\noteYW{Mention that our resolutions are not formulas and can be extended to the first-order cases naturally. }

% As we have proved by reduction, Inquisitive Logic of Knowing How is equally expressive as epistemic logic. In this section, we will give a direct translation from $\InqL$ to \SFive.

\begin{definition}[Resolutions \cite{Ciardelli2018}]
\
\begin{itemize}
	\item $\RR(p)=\{p\}$ for $p\in\Prop$
	\item $\RR(\bot)=\{\bot\}$
	\item $\RR(\alpha\vee\beta)=\RR(\alpha)\cup\RR(\beta)$
	\item $\RR(\alpha\wedge\beta)=\{\rho\wedge\sigma \mid \rho\in\RR(\alpha)$ and $\sigma \in\RR(\beta)\}$
	\item $\RR(\alpha\to\beta)=\{\bigwedge_{\rho\in\RR(\alpha)}(\rho\to f(\rho)) \mid f:\RR(\alpha)\to\RR(\beta)\}$
\end{itemize}
%where $\rho_{\nf}$ is some disjunction-free normal form of $\rho.$
\end{definition}
%From our perspective, the realizations defined above are the possible resolutions in terms of formulas. 
For example, instead of using the explicit functions for the resolutions of $\alpha\to\beta$, $\RR$ uses a conjunction of implications to capture the function. As proved in \cite{Ciardelli2011}, each element of $R(\alpha)$ is a \textit{statement}. 

The following proposition is stated without a proof in Proposition 9.3 in \cite{Ciardelli2011} which can be shown by an inductive proof. 
%\noteYW{I adopted the model in the following propositions. Need to double check. }
\begin{proposition}[\cite{Ciardelli2011}] \label{prop.realizesemantics}
For any model $\M$, state $s$ and formula $\alpha\in\PL$,
$$\M, s \Vdash \alpha \iff s \subseteq |\rho|_\M \text{ for some } \rho \in \RR(\alpha)
$$
where $|\rho|_\M$ is the set of possible worlds that satisfy $\rho$ classically in $\M$. 
\end{proposition}
% \begin{proof}
% \noteYW{Better provide the proof.}
% \end{proof}
In our know-how perspective, the above proposition actually establishes the equivalence of the state-based inquisitive semantics to the following alternative epistemic semantics for $\Kh\alpha$.
%based on Proposition \ref{prop.khreduction} (for the last \textit{iff}): 
$$\begin{array}{|lcl|}
\hline
\M,w\VDash \Kh\alpha&\iff& \text{ there exists a }
\rho \in \RR(\alpha) \text{ such that } \M,w\vDash \K\rho \\
\hline
\end{array}$$
% Now by Proposition \ref{prop.realizesemantics} and \ref{prop.satemodel}, we have:
% \begin{proposition}
% For any non-empty state $s$, and any $\alpha\in\PL$, $s\Vdash\alpha$ iff $\M_s\VDash\Kh\alpha$ iff $\M_s\Vdash\Kh\alpha$.
% \end{proposition}

% Note that for each $\alpha$, $\RR(\alpha)$ is finite, thus we can actually eliminate the quantifier in the above semantics by using a disjunction over all the possible realizations. This gives us the following truth-preserving translation:
% \begin{theorem}\label{thm.transs5}
% For any non-empty state $s$, and any $\alpha\in\PL$, $s\Vdash \alpha$ iff  $\M_s\vDash \bigvee_{\rho\in\RR(\alpha)}\K\rho$ iff $\M_s\vDash \Kh \bigvee_{\rho\in\RR(\alpha)}\rho$.         
% \end{theorem}
We can establish the following equivalences, without using Proposition \ref{prop.realizesemantics}. 
\begin{theorem}\label{thm.transs5}
	For any model $\M$ and any non-empty state $s$, and any $\alpha\in\LPL$, the following are equivalent: 
	\begin{enumerate}
	    \item $\M, s\Vdash\alpha $
	    \item $\M_s\vDash\Kh\alpha$ 
%	    \item $\text{There exists a } \rho \in \RR(\alpha) \text{ such that } \M_s\vDash \Kh\rho$ 
%	    \item $\M_s\vDash \bigvee_{\rho\in\RR(\alpha)}\Kh\rho$
	    \item $\M_s\vDash \bigvee_{\rho\in\RR(\alpha)}\K\rho$.
	    \item $\M_s\VDash \Kh\alpha$
	\end{enumerate}
	\end{theorem}
\begin{proof}
(1) iff (2) is due to Proposition \ref{prop.satemodel}. (3) iff (4) is based on  the definition of $\VDash$ and the fact that $\RR(\alpha)$ is finite. In the following, we show that  (2) iff (3) by induction on the structure of $\alpha$, where the definition of $\RR$ and classical reasoning play an important role. 
	\begin{itemize}
		\item $\alpha=p$ or $\alpha=\bot$: It is obvious since $\RR(\alpha)=\{\alpha\}$ and $\vDash\Kh\alpha\lra\K\alpha$ in such cases. 
 		\item $\alpha=\alpha_1\vee\alpha_2$: 
		$\M_s\vDash \Kh (\alpha_1\vee\alpha_2) 
		\iff \M_s\vDash \Kh \alpha_1\vee \Kh\alpha_2
		\iff \M_s\vDash \bigvee_{\rho\in\RR(\alpha_1)}\K\rho\vee \bigvee_{\rho\in\RR(\alpha_2)}\K\rho 
		\iff \M_s\vDash \bigvee_{\rho\in\RR(\alpha)}\K\rho$
		\item $\alpha=\alpha_1\wedge\alpha_2$:
		$\M_s\vDash \Kh (\alpha_1\wedge\alpha_2) 
		\iff \M_s\vDash \Kh \alpha_1\wedge \Kh\alpha_2 
		\iff \M_s\vDash \bigvee_{\rho\in\RR(\alpha_1)}\K\rho\wedge \bigvee_{\rho\in\RR(\alpha_2)}\K\rho
		\iff \M_s\vDash \bigvee_{\rho_1\in\RR(\alpha_1), \rho_2\in\RR(\alpha_2)}\K\rho_1\wedge\K\rho_2 
		\iff \M_s\vDash\bigvee_{\rho_1\in\RR(\alpha_1),\rho_2\in\RR(\alpha_2)}\K(\rho_1\wedge\rho_2) 
		\iff \M_s\vDash \bigvee_{\rho\in\RR(\alpha)}\K\rho$		
		\item $\alpha=\alpha_1\to\alpha_2$: 
		 $\M_s\vDash \Kh (\alpha_1\to\alpha_2) 
		\iff \M_s\vDash \K\Box(\Kh\alpha_1\to\Kh\alpha_2)$
	%	\iff \M_s\vDash \K\Box(\bigvee_{\rho\in\RR(\alpha_1)}\Kh\rho\to\bigvee_{\rho\in\RR(\alpha_2)}\Kh\rho)$. 

		Based on the induction hypothesis and the definition of $\RR(\alpha\to\beta)$ we just need to prove that: 
		
		$$\M_s\vDash \K\Box(\bigvee_{\rho\in\RR(\alpha_1)}\K\rho\to\bigvee_{\rho\in\RR(\alpha_2)}\K\rho) \iff \M_s\vDash \bigvee_{f:\RR(\alpha_1)\to\RR(\alpha_2)}\K\bigwedge_{\rho\in\RR(\alpha_1)}(\rho\to f(\rho)).$$
%		\noteYW{Should the conjunction be inside the scope of $\Kh$ on the right hand side?}

		$\Longrightarrow$: Given $\M_s\vDash \K\Box(\bigvee_{\rho\in\RR(\alpha_1)}\K\rho\to\bigvee_{\rho\in\RR(\alpha_2)}\K\rho)$, we need to find a  $f:\RR(\alpha_1)\to\RR(\alpha_2)$ s.t. $\M_s\vDash\K\bigwedge_{\rho\in\RR(\alpha_1)}(\rho\to f(\rho))$. 
		As $\RR(\alpha_2)$ is not empty, let $\rho_0$ be a fixed element of $\RR(\alpha_2)$ to be used to define the function $f$. There are two cases to be considered. 
		\begin{itemize}
		 \item 	For any $\rho\in\RR(\alpha_1)$ such that $\M_s\vDash\K\neg \rho$, we have  $\M_s\vDash\K(\rho\to\rho')$ for any $\rho'\in\RR(\alpha_2)$. We can safely define $f(\rho)=\rho_0$.
	    \item For any $\rho\in\RR(\alpha_1)$ such that $\M_s\vDash \hK \rho$, let $\M^{\rho}_s$ be the maximal submodel of $\M_s$ s.t. $\M^{\rho}_s\vDash \K\rho$. 
		By the semantics of $\K$ and $\Box$, if $\M_s\vDash \K\Box(\bigvee_{\rho\in\RR(\alpha_1)}\K\rho\to\bigvee_{\rho\in\RR(\alpha_2)}\K\rho)$, 
		$\M^{\rho}_s\vDash\bigvee_{\rho\in\RR(\alpha_1)}\K\rho\to\bigvee_{\rho\in\RR(\alpha_2)}\K\rho$, 
		then $\M^{\rho}_s\vDash\bigvee_{\rho\in\RR(\alpha_2)}\K\rho$, 
		which means that there is a $\rho'\in\RR(\alpha_2)$ such that $\M^{\rho}_s\vDash\K\rho'$. 
		As $\M^{\rho}_s$ is the maximal submodel of $\M_s$ s.t. $\M^{\rho}_s\vDash \K\rho$, 
		$\M_s\vDash\K(\rho\to\rho')$. We let $f(\rho)=\rho'$. 
		\end{itemize}
	   Now we have defined an $f:\RR(\alpha_1)\to\RR(\alpha_2)$ s.t. $\M_s\vDash\bigwedge_{\rho\in\RR(\alpha_1)}\K(\rho\to f(\rho))$, 
		which is equivalent to $\M_s\vDash\K\bigwedge_{\rho\in\RR(\alpha_1)}(\rho\to f(\rho))$.

		$\Longleftarrow$: 
		%If $\M_s\vDash \bigvee_{f:\RR(\alpha_1)\to\RR(\alpha_2)}\K\bigwedge_{\rho\in\RR(\alpha_1)}(\rho\to f(\rho))$, 
		%there is a $f:\RR(\alpha_1)\to\RR(\alpha_2)$ such that $\M_s\vDash\K\bigwedge_{\rho\in\RR(\alpha_1)}(\rho\to f(\rho))$, 
		%which is equivalent to $\M_s\vDash\bigwedge_{\rho\in\RR(\alpha_1)}\K(\rho\to f(\rho))$. 
		%So for any submodel $\M'$ of $\M_s$, $\M'\vDash\bigwedge_{\rho\in\RR(\alpha_1)}(\K(\rho\to f(\rho)))$, 
		%and by \DISTK: $\M'\vDash\bigwedge_{\rho\in\RR(\alpha_1)}(\K\rho\to \K f(\rho))$. 
		%So $\M'\vDash\bigvee_{\rho\in\RR(\alpha_1)}\K\rho\to\bigvee_{\rho\in\RR(\alpha_2)}\K\rho$ and $\M_s\vDash \K\Box(\bigvee_{\rho\in\RR(\alpha_1)}\K\rho\to\bigvee_{\rho\in\RR(\alpha_2)}\K\rho)$. 
		Suppose $\M_s\vDash \bigvee_{f:\RR(\alpha_1)\to\RR(\alpha_2)}\K\bigwedge_{\rho\in\RR(\alpha_1)}(\rho\to f(\rho))$ then there is a $f:\RR(\alpha_1)\to\RR(\alpha_2)$ s.t. $\M_s\vDash\K\bigwedge_{\rho\in\RR(\alpha_1)}(\rho\to f(\rho))$. This amounts to $\M_s\vDash\bigwedge_{\rho\in\RR(\alpha_1)}\K(\rho\to f(\rho))$, thus 
		for any submodel $\M'$ of $\M_s$, $\M'\vDash\bigwedge_{\rho\in\RR(\alpha_1)}(\K(\rho\to f(\rho)))$. By the usual distribution axiom by of $\K$,  
		$\M'\vDash\bigwedge_{\rho\in\RR(\alpha_1)}(\K\rho\to \K f(\rho))$. Weakening the consequent, we have 
		$\M'\vDash\bigwedge_{\rho\in\RR(\alpha_1)}(\K\rho\to \bigvee_{\rho\in\RR(\alpha_2)}\K\rho).$ Therefore 
		$\M'\vDash\bigvee_{\rho\in\RR(\alpha_1)}\K\rho\to\bigvee_{\rho\in\RR(\alpha_2)}\K\rho$. Since $\M'$ is an arbitrary submodel of $\M_s$, it follows that  
		$\M_s\vDash \K\Box(\bigvee_{\rho\in\RR(\alpha_1)}\K\rho\to\bigvee_{\rho\in\RR(\alpha_2)}\K\rho).$
		%If we want to prove $\M_s\vDash \K\Box(\bigvee_{\rho\in\RR(\alpha_1)}\K\rho\to\bigvee_{\rho\in\RR(\alpha_2)}\K\rho)$, we only need to prove that $\M'\vDash\bigvee_{\rho\in\RR(\alpha_1)}\K\rho\to\bigvee_{\rho\in\RR(\alpha_2)}\K\rho$ for any submodel $\M'$ of $\M_s$. By $\M_s\vDash\bigwedge_{\rho\in\RR(\alpha_1)}(\K(\rho\to f(\rho)))$, we know that $\M'\vDash$
	\end{itemize}
\end{proof}
As an immediate consequence, we have the modal translation from $\InqL$ to the epistemic logic $\SFiveL$ and the modal logic $\SK$.  The translation to modal logic $\SK$ (and other normal modal logics) was mentioned by Ciardelli in \cite[Section 6.6]{Ciardelli16phd}, and also in \cite[Section 5.4]{Ciardelli2018} when discussing the modal approach of \cite{Nelken2006} to the semantics of questions. Through the correspondence with the know-how semantics (w.r.t.\ $\VDash$) shown in the above theorem, we can see more intuitively what the translation below is doing.

\begin{corollary}[$\InqL$ to $\SFiveL$ and $\SK$ \cite{Ciardelli16phd,Ciardelli2018}] 
    	For $\alpha\in\LPL$, $\alpha\in\InqL$ iff $\bigvee_{\rho\in\RR(\alpha)}\K\rho\in \SFiveL$ iff $\bigvee_{\rho\in\RR(\alpha)}\K\rho\in \SK$
\end{corollary}
\begin{proof}
    The last \textit{iff} is due to the fact that for each pointed Kripke model $\M,w\vDash\neg \bigvee_{\rho\in\RR(\alpha)}\K\rho$ there is an epistemic model $\N, w\vDash \neg \bigvee_{\rho\in\RR(\alpha)}\K\rho$ where $\N=\lr{W_{\M}, \sim, V_\M}$ and $\sim$ is the reflexive,  symmetric, transitive closure of the accessibility relation in $\M$. Note that $\neg \bigvee_{\rho\in\RR(\alpha)}\K\rho$ is equivalent to $\bigwedge_{\rho\in \RR(\alpha) }\hK\neg\rho.$
\end{proof}

%\noteYW{Mantion the parellel between classical and epistemic reaosning. }

Note that although it is natural to think this translation can be compared to G\"odel's translation of intuitionistic logic to $\SFour$, the nature of the two translations are quite different in terms of how to read the modality. In G\"odel's translation, the modality is actually a temporal-epistemic one, namely $\K\Box$ in our perspective \cite{Wang2021}, but the modality here is merely a purely epistemic one.  

As another corollary of Theorem \ref{thm.transs5}, the following important property of $\InqL$ follows.
\begin{corollary}[\cite{Ciardelli2011,Ciardelli2016}]
Any $\alpha\in\PL$ is equivalent to a disjunction of statements\slash negations  in inquisitive logic.
\end{corollary}
\begin{proof}
From Theorem \ref{thm.transs5}, $\vDash\Kh\alpha \lra  \bigvee_{\rho\in\RR(\alpha)}\K\rho$. Due to Proposition  \ref{prop.disjfree}, $\rho\in \RR(\alpha)$ is a statement. Therefore by Proposition \ref{prop.assert}, $\vDash\Kh\alpha \lra  \bigvee_{\rho\in\RR(\alpha)}\Kh \rho$, thus   $\vDash\Kh\alpha \lra  \Kh \bigvee_{\rho\in\RR(\alpha)}\rho$. By the validity of the Rule $\RKhI$, we have $\vDash \Kh (\alpha \lra \bigvee_{\rho\in\RR(\alpha)}\rho).$ By Corollary  \ref{coro.equiv}, we have $\alpha\lra \bigvee_{\rho\in\RR(\alpha)}\rho\in \InqL$.   % It is not hard to show that $\vDash \Kh() $    
%(3) iff (4) is due to Proposition \ref{prop.assert} and the fact that $\rho$ is a statement based on Proposition \ref{prop.disjfree}.
\end{proof}
%\noteYW{to be checked and polished}

The above corollaries show that inquisitive logic can be viewed as a fragment of normal epistemic logic technically. However, as argued in \cite{Ciardelli2018} regarding the modal approach of \cite{Nelken2006}, those modal formulas do not preserve the surface structure of sentences and questions in the natural language. Actually, as shown by our approach, such a reduction to standard epistemic logic is the result of the assumption in inquisitive logic that each atomic proposition has a unique resolution. In similar settings such as intuitionistic logic and Medvedev logic, this is not possible. Using our powerful language $\InqKhL$, we can keep the structure of statements and questions as they are in the natural language by $\Kh\alpha$ on the one hand, and reveal its epistemic meaning by the reductions on the other hand. We do not need to take sides between the technical and conceptual convenience. 

Yet another consequence of Theorem \ref{thm.transs5} is about the limitation of inquisitive logic over models.
\begin{corollary}
Inquisitive logic is less expressive than $\LEL$.  
\end{corollary}
\begin{proof}
Note that inquisitive logic can only say things in terms of disjunction over positive $\K\alpha$ formulas, thus formulas like $\hK p$ are simply not expressible. 
\end{proof}

With the classical connectives in hand, we can express various things which were not expressible in the standard inquisitive logic, such as the classical negation of an inquisitive formula $\alpha$ (simply by $\neg\Kh\alpha$) (cf. \cite{Puncochar2015} for the study within the framework of inquisitive logic). The classical connectives and normal know-that modality give us lots of flexibility in capturing mixed reasoning with both inquisitive and classical propositions.

% Ciardelli argues against the modal approach to the modal semantics of question in \cite{Nelken2006} with the following points: 

% \begin{itemize}
%     \item \textbf{Evaluating to a single world is not necessary.} This is not convincing, as in our setting we can express $\alpha\land \neg \K\alpha$, $\K\alpha \land \neg\Kh\alpha$ and so on. We can also express $\hK \alpha$ with the classical negation. 
%     \item \textbf{Insight on dependence} Our approach make the dependence even more explicitly. The Armstrong axioms can be expressed intuitively in our approach. 
%     \item \textbf{Preserving surface structure in natural language} We preserve it and decompose it. If Alice comes to the party, Box will come to the party: $p\to q$. If Alice comes to the party, will Bob comes to the party too? $p\to ?q$. 
%    \item \textbf{Computational interpretation} We have a constructive interpretation in the modal approach. 
%(1) Always D(x, x).
% (2) If D(x, y, z), then D(y, x, z).
% (3) If D(x, x, y), then D(x, y).
% (4) If D(x, z), then D(x, y, z).
% (5) If D(x, y) and D(y, z), then D(x, z).
% (1) p → p.
% (2) (p →q →r ) → (q → p →r ).
% (3) (p → p →q) → (p →q).
% (4) (p →r ) → (p →q →r ).
% (5) (p →q) → (q →r ) → (p →r ).
%     \item[] 
% \end{itemize}
\subsection{Comparison with inquisitive modal logic}

In the literature, \textit{inquisitive modal logic} is proposed as a conservative extension of normal modal logic by introducing the modality in the language and separating the inquisitive disjunction $\inqlor$ with the classical disjunction $\lor$ (cf. e.g., \cite{Ciardelli16phd}). To avoid potential confusion with our $\Box$ modality, we denote the inquisitive modality as $\boxdot$ in this subsection. With $\boxdot$ at hand, one can naturally express formulas combining the modality and questions, such as $\boxdot (p\inqlor \neg p)$, which says \textit{knowing whether $p$} in an epistemic setting. Correspondingly, the information model is extended with a binary relation $\Rel$, not to be confused with the resolution function $R$, to interpret the modality with the following extra semantic clause in Definition 6.1.3 in \cite{Ciardelli16phd}: 
$$\begin{array}{|lcl|}
\hline
\M,s\Vdash \boxdot\alpha&\iff& \text{for all $w \in s$, $\M, \Rel[w] \Vdash \alpha$}\\
\hline
\end{array}$$
\noindent where $\Rel[w]=\{ v \mid w\Rel v \text{ in $\M$}\}$.
When $s$ is a singleton set, we can derive the following semantics for $\boxdot$ over \textit{worlds} (Proposition 6.1.4 in \cite{Ciardelli16phd}): 
$$\begin{array}{|lcl|}
\hline
\M,w\Vdash \boxdot\alpha&\iff& \M, \Rel[w] \Vdash \alpha\\
\hline
\end{array}$$
By the above support-based semantics, $\boxdot\alpha$ is always truth-conditional, i.e., $s$ supports $\boxdot\alpha$ iff it is
true at each world $w \in s$ (cf. \cite{Ciardelli16phd}). When $\alpha$ is also truth-conditional, the above semantics of $\boxdot\alpha$ over worlds clearly boils down to the Kripke semantics for $\boxdot$: $\boxdot\alpha$ is true at $w$ if $\alpha$ is true at each $v$ accessible from $w$. In \cite{Ciardelli16phd}, a general method of axiomatizating the logic is provided that works with various frame conditions. 

Since this is another way of combining modalities and inquisitive connectives, it deserves a comparison with our approach, in particular about the semantics of $\boxdot$ under an epistemic reading and our $\Kh$. 

By definition, $\boxdot$ relies on the support relation, while $\Kh$ has the $\exists\K$-shape semantics that relies on the resolutions as in Definition \ref{def.resolution}. However, they are connected deeply and come down to the same truth condition on pointed models. To see this, first let us consider the special case where $\Rel$ is the total-relation in a model $\M$ for inquisitive modal logic. Therefore, for any $w\in W_\M$, $\Rel[w]$ is exactly the set of possible worlds $W_\M$. Given $\M$, let $s_t$ be the trivial state $W_\M$, and let $w\in s_t$, the support semantics of $\boxdot\alpha$ over $s_t$ collapses into the one for the non-modal $\alpha$:
$$\begin{array}{|lcl|}
\hline
\M,s_t\Vdash \boxdot\alpha&\iff& \M, s_t \Vdash \alpha\iff \M,w\Vdash\boxdot\alpha\\
\hline
\end{array}$$
Note that by Lemma \ref{lem.modelstate}, we have: 
$$\begin{array}{|lcl|}
\hline
\M,s_t\Vdash \alpha&\iff& \M \vDash \Kh \alpha\iff \M,w\vDash \Kh\alpha\\
\hline
\end{array}$$
From the above two observations, we can establish that for any $\alpha\in\LPL$:
$$\begin{array}{|lcl|}
\hline
\M,w\Vdash \boxdot\alpha& \iff& \M,w \vDash\Kh\alpha\\
\hline
\end{array}$$
The above equivalence still holds on models with an arbitrary relation $\Rel$, \footnote{It is observed by the anonymous reviewer.} not just the total one,  given that we generalize our $\Kh$ semantics naturally to models with $\Rel$, as in  logics of \textit{know-wh} (cf. e.g.,  \cite{Wang2018}). 
$$\begin{array}{|lcl|}
\hline
\M,w\models \Kh\alpha&\iff& \text{ there exists an }  x \text{ s.t. for any }  v\in \Rel[w], x\in\R(v,\alpha)\\
\hline
\end{array}$$
Under this more general semantics of $\Kh$ over models with an arbitrary relation $\Rel$, we can establish: 
$$\begin{array}{|lll|}
\hline
\M,w\Vdash \boxdot\alpha&\iff \M,R[w]\Vdash \alpha&\iff \M,w\models\Kh\alpha\\
\hline
\end{array}$$
The second equivalence is again an application of Lemma \ref{lem.modelstate}  under the help of Proposition \ref{prop.irre}. Essentially, the equivalence is due to the fact that the intended function of $\Kh$ is exactly to capture the support relations between a state and an inquisitive formula, and the truth conditional semantics of $\boxdot$ gave $\Rel[w]$ as such a state. Thus the two different routes that $\Kh$ and $\boxdot$ take in their apparent differently semantics converge to the same truth condition eventually. This equivalence is also suggested by Propositions \ref{prop.realizesemantics} and Theorem \ref{thm.transs5} showing that the support semantics can be viewed as a resolution-based semantics. However, in general, it is not guaranteed that the (re)solution-based semantics, such as the one for Medvedev's logic, can be transformed into an equivalent state-based semantics over information models, for the structure of resolutions for (atomic) propositions is richer than the mere truth values on each world. In the case of inquisitive logic, the situation is very much simplified by assuming the atomic propositions have one and only possible resolution.

%Note that the fact that $\boxdot$ collapses under the epistemic setting may also suggest there is some modal information already in the support semantics for non-modal formulas. 

There is also an interesting mismatch between our work and inquisitive modal logic. In our setting, we explicitly separate the different roles of the \textit{modalities} by using both $\K$ and $\Kh$. On the other hand, we use the same symbol $\lor$ for inquisitive and classical \textit{disjunctions}. In contrast, in inquisitive modal logic, the modality has a double role to play depending on what is in the scope, while the two disjunctions are differentiated explicitly by distinct symbols. 

The difference is firstly conceptual. From our epistemic perspective, the non-classical behavior of logical connectives are due to the implicit epistemic modality $\Kh$. From the more recent perspective of inquisitive logic \cite{Ciardelli16phd}, the separation of the two disjunctions makes it more clear that the inquisitive logic can be viewed as an extension of classical logic with the inquisitive operators. The difference also leads to various  logical properties. For example, since we use the same symbol $\lor$ for these two disjunctions, it can function differently inside and outside the scope of $\Kh$. It follows that the usual unconditional rule of \textit{replacement of equals} in the scope of a modality is invalid, thus making $\Kh$ a \textit{hyperintensional operator}. On the other hand, by having $\K$ and $\Kh$ explicitly, we can differentiate them in the axioms and reveal how negation and atomic propositions can act as bridges to connect the two types of knowledge, as in $\K p\lra\Kh p$ and $\Kh\neg \alpha\lra \K\neg\alpha$. Note that, thanks to our choice of using the same $\lor$ symbol, the axioms such as $\Kh\alpha\to\K \alpha$ can be written in a natural way, without introducing unnecessary translations of formulas. We think each approach has its features and advantages. In particular, the inquisitive modality has the very elegant feature of deriving the semantics of statements in terms of \textit{know+embedded} questions compositionally (cf. \cite[Section 6.2]{Ciardelli16phd}). The combination of the two approaches can be explored in the future. 
%\noteW{this part might be hard for others to understand, we mentioned a lot of things here: the relation between $\K$ and $\Kh$, the relation between two disjunctions and their relation to inquisitve modality. I think these contents can be hardly stated in a few sentences, and there are not enough relevant results in the previous sections to support us to clearly give their relationship.} 

Beyond the above connections and differences, $\Kh$ and $\boxdot$ are very different in the motivation behind them. The point of our $\Kh$ operator is to capture the epistemic content \textit{already} in the support semantics, while the modality in inquisitive modal logic is to \textit{add} the modal information into the picture. Moreover, in our framework, further  modalities and connectives are used to ``open up'' the $\Kh$ formulas to reveal their intuitive epistemic readings. In a nutshell, we want to turn the inquisitive formulas into classical ones with also epistemic and dynamic modalities to obtain intuitive epistemic readings of them. For example, our approach also features a dynamic modality $\Box$ to open up  $\Kh(\alpha\to\beta)$ by the equivalent $\K\Box (\Kh\alpha\to \Kh\beta)$, i.e., knowing how $\alpha$ implies $\beta$ means knowing that whenever one knows how $\alpha$ is true, one also knows how $\beta$ is true. All these extra modalities and the classical connectives helped us to ``decode'' the non-classical behaviors of the inquisitive logic, from our epistemic perspective. As our intuitive axioms showed, the information behind the innocent-looking propositional formulas of inquisitive logic is very rich, under the epistemic view of point.    
%\noteW{looks weird, maybe ``whenever one knows how $\alpha$ then one knows how $\beta$ is true'' or ``whenever one knows how $\alpha$, one knows how $\beta$ is true'' is more readable} 
 Interestingly, as we also showed in the paper, $\Box$ and $\Kh$ can be eliminated eventually. 

There is one more distinction to be pointed out. In our setting, the modalities cannot occur in the scope of $\Kh$ operators, but this may not be an essential restriction (at least for the modality $\K$), given the discussion on the resolutions for inquisitive modalities (cf. \cite[Section 6.3]{Ciardelli16phd}). We leave the study of the extended language for a future occasion.\footnote{We also thank the anonymous reviewer for pointing out this to us. }

%Conceptually, our approach is also quite different from inquisitive modal logic. the starting point of ours is not adding modality to the inquisitive logic, but to ``open up'' the existing inquisitive formulas from an epistemic perspective using a more powerful language to obtain some intuitive meaning. 

%\noteYW{Discuss $\Box ?p$ and $\K \Kh ?p$}

\medskip

As another incarnation of inquisitive modal logic with dynamics,  \textit{Inquisitive Dynamic Epistemic Logic} ($\InqDEL$) is proposed and studied in \cite{CiardelliR15,VanGessel2020}. Since our approach is also  \textit{dynamic-epistemic} in nature, it also deserves some comparison with $\InqDEL$. Again, our approach is to reveal the dynamic-epistemic structure implicitly in the existing inquisitive logic from the epistemic perspective, whereas $\InqDEL$ extends the version of inquisitive semantics with extra epistemic structures and dynamics. In the models of $\InqDEL$, there is an \textit{issue function} $\Sigma$  assigning each possible world an \textit{issue} $\Sigma(w)$, i.e., a non-empty, downward closed set of states, satisfying some intuitive epistemic conditions. At each world $w$, the set of epistemically indistinguishable worlds $\sigma(w)$ is then defined as $\bigcup \Sigma(w).$ Moreover, the work of \cite{CiardelliR15} is based on a specific version of inquisitive semantics in \cite{Ciardelli2015} where a dichotomous syntax, distinguishing declarative and interrogative sentences is used instead of the unified framework of \cite{Ciardelli2011}. On top of this dichotomous language, the know-that modality and an extra modality of \textit{entertain} are added, where the latter modality can describe the issue currently in concern.\footnote{A similar dynamic epistemic approach handling issues is \cite{VanBenthem2012} (cf. \cite{CiardelliR15} for a detailed comparison between \cite{CiardelliR15} and \cite{VanBenthem2012}).} In contrast, as we mentioned, the dynamic operator $\Box$ in our setting is merely to capture the inquisitive implication. In the model, we also do not have the structures of issues.

\section{Conclusions}\label{sec.conc}
This paper is a case study of the general research programme proposed in \cite{Wang2021} on ``epistemicizating'' intuitionistic logic and its relatives. We showed that, as an alternative interpretation, the propositional inquisitive logic $\InqL$ can be viewed as a (dynamic) epistemic logic of knowing how over standard S5 epistemic models. In our approach, an inquisitive formula $\alpha$ being supported by a state $s$ is formalized as it is known how to resolve $\alpha$ (or simply knowing how $\alpha$ is true). We start by turning an inquisitive formula $\alpha$ into the equivalent know-how formula $\Kh\alpha$ in our framework. Then by using modalities of know-that $\K$ and informational updates $\Box$ based on classical connectives, we can \textit{unload} the epistemic contents hidden in such know-how formulas $\Kh\alpha$ by reducing the complexity of $\alpha$ step by step. From the point of view of the general programme of \cite{Wang2021}, $\InqL$ is a particularly interesting case since the corresponding know-how modality can be eventually eliminated based on the fact that the resolution of each atomic proposition is \textit{unique}, which is the reason that the axiom $\neg\neg p\to p$ holds for atomic propositions $p$ in inquisitive logic. In our framework, it amounts to the crucial axiom $\K p \to \Kh p$, i.e., \textit{knowing that} $p$ is true implies the apparently stronger \textit{knowing how} it is true, which can help to reduce the know-how operator eventually. Given such a simplification, technically, we can view inquisitive logic as a fragment of standard epistemic logic. 
\medskip

What we have presented so far is clearly only the beginning of an interesting story regarding the classical \textit{``epistemicization''} of the intuitionistic logic and its friends. Here we just list a few further directions. 
%\noteYW{Need to think whether we should mention all of these further directions.}

\begin{itemize}
\item Given the close connections between inquisitive logic and dependence logic (cf. e.g., \cite{Ciardelli2016}), it is definitely interesting to see how we can give epistemic interpretations of dependence logic in various forms, where in the semantics a \textit{team} can be viewed as an epistemic model. Note that as observed by \cite{CiardelliB19}, the crucial notion of disjunction in dependence logic, i.e., the tensor, is not directly definable in inquisitive logic. This also presents a challenge to the epistemicization of dependence logic in our framework since we would like to unload the epistemic content of tensor disjunction in a compositional manner. See \cite{tensor22} for the first attempt. 
\item As in intuitionistic logic, quantifiers may bring extra complications, so it is interesting to see how we can extend our work to capture the first-order inquisitive logic \cite{Ciardelli09}. 
\item Our approach also opens a natural way to extend inquisitive logic to a multi-agent setting. However, our reduction of the dynamic operator relies on the single-agent setting. It remains to see whether we can still reduce the multi-agent epistemicization of inquisitive logic to the multi-agent S5. 
\item The interpretation of inquisitive formulas in our setting as \textit{knowing how $\alpha$ is true} has an obvious connection with truthmaker semantics based on the  intuitionistic spirit \cite{Fine14}. The exact connections invite close investigations. 
\item It is also interesting to see how we can combine the idea of inquisitive modal logics with our work to benefit from both frameworks. 
%\item Include the tensor $\otimes$ in dependence logic, we suspect that $\otimes$ is also not definable in 
\end{itemize}

\paragraph{Acknowledgement}
The authors thank Hans van Ditmarsch for pointing out the reduction of $\Box$ in single-agent Arbitrary Public Announcement Logic, and thank Feng Ye and Shengyang Zhong for pointing out a gap in a proof of a previous version. The authors also would like to thank  Ivano Ciardelli and Fan Yang for the connections and relevant results about propositional dependence logic and inquisitive modal logic. The authors are very grateful to the reviewer of this journal whose insightful comments helped in improving the presentation of the paper. Yanjing Wang acknowledges the support from the NSSF grant 19BZX135. 

\bibliographystyle{plain}

\appendix
\section{Remaining proof for Proposition \ref{prop.box}}
\label{sec.app}

%\noteYW{To be checked and updated with the new notions. }\\
\begin{proof}
In the following, let $\SFive_{\K}$ be the proof system of $\SFiveL$ including $\AxTK,\AxTransK$, and $\AxEucK$.
\text{For \BD:}
{\setcounter{equation}{0}}
\begin{align}
{\setcounter{equation}{0}}
	&\vdash\alpha\lra\Box\alpha&&\Per,\AxTB, \label{1}\\
	&\vdash\alpha\lor\Box\varphi\lra\Box\alpha\lor\Box\varphi&&(\ref{1})\rRE\label{2}\\
	&\vdash\Box\alpha\lor\Box\varphi\to\Box(\alpha\lor\varphi)&& \SFour_{\Box}\label{3}\\
	&\vdash\alpha\lor\Box\varphi\to\Box(\alpha\lor\varphi)&& (\ref{2})(\ref{3})\rRE\label{8}\\
	&\vdash\Box(\alpha\lor\varphi)\to(\alpha\land\Box(\alpha\lor\varphi))\lor(\neg\alpha\land\Box(\alpha\lor\varphi)) &&\TAUT\label{4}\\
	&\vdash\neg\alpha\lra\Box\neg\alpha&&\Per, \AxTB\label{10}\\
	&\vdash(\neg\alpha\land\Box(\alpha\lor\varphi))\leftrightarrow(\Box\neg\alpha\land\Box(\alpha\lor\varphi)) &&(\ref{10})\rRE\label{5}\\
	&\vdash(\Box\neg\alpha\land\Box(\alpha\lor\varphi))\leftrightarrow\Box(\neg\alpha\land(\alpha\lor\varphi)) &&\SFour_{\Box}\label{6}\\
	&\vdash\Box(\neg\alpha\land(\alpha\lor\varphi))\leftrightarrow(\Box\neg\alpha\land\Box\varphi) &&\SFour_{\Box}\label{7}\\
	&\vdash(\neg\alpha\land\Box(\alpha\lor\varphi))\to\Box\varphi &&(\ref{5})(\ref{6})(\ref{7})\TAUT\label{12}\\
	&\vdash\Box(\alpha\lor\varphi)\to(\alpha\land\Box(\alpha\lor\varphi))\lor\Box\varphi&&(\ref{4})(\ref{12})\SFour_{\Box}\label{11}\\
	&\vdash\Box(\alpha\lor\varphi)\to\alpha\lor\Box\varphi &&(\ref{11})\TAUT\label{9}\\
	&\vdash\Box(\alpha\lor\varphi)\leftrightarrow \alpha\lor\Box\varphi&&(\ref{8})(\ref{9})\TAUT
\end{align}
 \text{ For \BKD}: Let $\phi$ be  ($\hK\alpha\lor\K\alpha_1\lor\dots\lor\K\alpha_n$) and let $\psi$ be ($\alpha\lor\K(\alpha\lor\alpha_1)\lor\dots\lor\K(\alpha\lor\alpha_n)$) below. 
 {\setcounter{equation}{0}}
\begin{align}
{\setcounter{equation}{0}}
	&\vdash\alpha\to\Box\alpha&&\Per\label{22}\\
	&\vdash\Box\alpha\to\Box\hK\alpha&& \SFive_{\K},\SFour_{\Box}\label{23}\\
	&\vdash\alpha\to\Box\hK\alpha&& (\ref{22})(\ref{23})\TAUT\label{24}\\
	&\vdash\alpha\to\Box\phi&& (\ref{24})\SFive_{\K},\SFour_{\Box}\label{26}\\
	&\vdash\alpha\to(\psi\to\Box\phi)&& (\ref{26})\TAUT\label{27}\\
	&\vdash(\alpha\lor\alpha_1)\to\Box(\alpha\lor\alpha_1)&& \Per\label{29}\\
	&\vdash\K(\alpha\lor\alpha_1)\to\K\Box(\alpha\lor\alpha_1)&& (\ref{29})\SFive_{\K}\label{30}\\
	&\vdash\K\Box(\alpha\lor\alpha_1)\to\Box\K(\alpha\lor\alpha_1)&& \PR\label{31}\\
	&\vdash\K(\alpha\lor\alpha_1)\to\Box\K(\alpha\lor\alpha_1)&& (\ref{30})(\ref{31})\TAUT\label{32}\\
	&\vdash\K(\alpha\lor\alpha_1)\to(\hK\alpha\lor\K\alpha_1)&& \SFive_{\K}\label{33}\\
	&\vdash\Box\K(\alpha\lor\alpha_1)\to\Box(\hK\alpha\lor\K\alpha_1)&& (\ref{33})\SFour_{\Box}\label{34}\\
	&\vdash\K(\alpha\lor\alpha_1)\to\Box(\hK\alpha\lor\K\alpha_1)&& (\ref{32})(\ref{34})\TAUT\label{35}\\
	&\vdash\K(\alpha\lor\alpha_1)\to\Box\phi&& (\ref{35})\SFour_{\Box}\label{36}\\
	&\vdash\K(\alpha\lor\alpha_i)\to\Box\phi\ (i\in\{1,\dots,n\})&& (\ref{36})\label{37}\\
	&\vdash\psi\to\Box\phi&& (\ref{26})(\ref{37})\TAUT\label{38}\\
    &\vdash\neg\alpha\land\!\bigwedge_{1\leq i\leq k}\hK(\neg\alpha\land\neg\alpha_i)\!\to\!\Diamond(\K\neg\alpha\land\!\bigwedge_{1\leq i\leq k}\hK\neg\alpha_i)
    &&\EU_k\label{60}\\
    &\vdash\neg\Diamond(\K\neg\alpha\land\!\bigwedge_{1\leq i\leq k}\hK\neg\alpha_i)\!\to\!\neg(\neg\alpha\land\!\bigwedge_{1\leq i\leq k}\hK(\neg\alpha\land\neg\alpha_i))
    &&(\ref{60})\TAUT\label{61}\\
    &\vdash\Box(\hK\alpha\lor\!\bigvee_{1\leq i\leq k}\K\alpha_i)\!\to\!(\alpha\lor\!\bigvee_{1\leq i\leq k}\K(\alpha\lor\alpha_i))
    &&(\ref{61}) \SFour_{\Box}\label{62}\\
    &\vdash\Box\phi\to\psi
    &&(\ref{62}) \label{63}\\
	&\vdash\Box\phi\leftrightarrow\psi 
	&&(\ref{38})(\ref{63})\MP
\end{align}
\end{proof}
\end{document}